\documentclass[11pt,a4paper]{article}
\usepackage[utf8]{inputenc}
\usepackage{relsize}
\usepackage[T1]{fontenc}
\usepackage{microtype}
\usepackage{amsmath, amsthm}
\usepackage{amsfonts}
\usepackage{graphicx}
\usepackage[english]{babel}
\usepackage[nottoc]{tocbibind}
\usepackage{mathrsfs}
\usepackage{amssymb}
\usepackage{mathtools}
\usepackage[poly,arrow,curve,matrix]{xy}
\usepackage[a4paper]{geometry}
\usepackage{mathabx}

\usepackage{stmaryrd}

\usepackage[colorlinks]{hyperref}
\usepackage{geometry}
\usepackage{enumerate}
\usepackage[center]{caption}

\usepackage{tikz-cd}
\usepackage{lipsum}
\usepackage{adjustbox}
\usepackage{tikz}

\def\restrict#1{\raise-.5ex\hbox{\ensuremath|}_{#1}}

\def\XXint#1#2#3{{\setbox0=\hbox{$#1{#2#3}{\int}$ }
\vcenter{\hbox{$#2#3$ }}\kern-.5775\wd0}}

\def\Spec{\mathop{\mbox{\normalfont Spec}}\nolimits}

\newtheorem{theorem}{Theorem}[section]
\newtheorem{lemma}[theorem]{Lemma}
\newtheorem{Sublemma}[theorem]{Sub-lemma}
\newtheorem{proposition}[theorem]{Proposition}

\newtheorem{corollary}[theorem]{Corollary}
\newtheorem{remark}[theorem]{Remark}
\newtheorem{definition}[theorem]{Definition}

\DeclareMathOperator{\Sh}{Sh}

\renewcommand{\Spec}[1]{\operatorname{Spec}(#1)}
\newcommand{\C}{\mathbb{C}}
\newcommand{\Q}{\mathbb{Q}}
\newcommand{\A}{\mathbb{A}}

\newcommand{\Z}{\mathbb{Z}}
\newcommand{\R}{\mathbb{R}}

\title{G-companions on algebraic stacks and applications to canonical $l$-adic local systems on Shimura stacks}
\author{Min Shi}
\date{}

\begin{document}

\maketitle
\begin{abstract}
\noindent
Cases of Deligne’s companion conjecture for normal schemes over finite fields have been proven by L. Lafforgue, Drinfeld, and Zheng in recent years: L. Lafforgue proved the conjecture for curves, Drinfeld proved the conjecture for all smooth schemes and later also for representations valued in a reductive group, and Zheng proved Deligne's conjecture for smooth Artin stacks. In this paper, we extend Drinfeld’s theorem for general reductive groups to smooth Artin stacks
(see definition below) of finite presentation, and apply the result to the study of compatibility of the canonical $\ell$-adic local systems on Shimura stacks.
\end{abstract}
\section{Introduction}
Let $\mathbb{F}_q$ be a finite field, and let $\ell$ and $\ell'$ be prime numbers distinct from the characteristic $p$ of $\mathbb{F}_q$. We let $\overline{\mathbb{Q}}_\ell$ denote an algebraic closure of $\mathbb{Q}_\ell$. Deligne conjectured \cite[Conjecture 1.2.10]{Deligne} that for any irreducible lisse $\overline{\mathbb{Q}}_{\ell}$-sheaf $\mathcal{F}$ with finite determinant on a normal variety $X$ over $\mathbb{F}_q$, pure of weight $0$, and for any prime $\ell' \neq \mathrm{char}(\mathbb{F}_q)$, there exists a lisse $\overline{\mathbb{Q}}_{\ell'}$-sheaf $\mathcal{F}'$ such that $\mathcal{F}$ and $\mathcal{F}'$ have the same Frobenius characteristic polynomials at every closed point of $X$. Lafforgue proved one case of the global Langlands for $X/\mathbb{F}_q$ where $X$ is a curve; Drinfeld proved Deligne's conjecture for smooth schemes, and constructed $G$-companions for all reductive (not necessarily connected) groups (see \textsection 2.2 for definition). The first goal of this paper is to extend Drinfeld’s theorem to smooth algebraic stacks of finite presentation, which gives the first main result:
\begin{theorem}[see Theorem \ref{theorem 5.4}]
    Let $\mathcal{X}$ be a smooth algebraic stack of finite presentation over $\mathbb{F}_q$. Let $\lambda$ and $\lambda'$ be non-Archimedean places of $\overline{\Q}$ not dividing $p$. Let $G$ be a (not necessarily connected) reductive group over $\Q$. Let $x_0$ be a closed point of $\mathcal{X}$ and let $\Pi$ denote the \'etale fundamental group of $\mathcal{X}$ based at $\overline{x_0}$. Let $\rho_{\lambda}:\Pi\to G(\overline{\Q_{\lambda}})$ be a $G$-irreducible representation. Let $\rho^{ab}_{\lambda}$ be the map $\Pi\xrightarrow{\rho_{\lambda}} G(\overline{\Q_{\lambda}})\to (G/[G^{\circ},G^{\circ}])(\overline{\Q_{\lambda}})$ where the latter map is the canonical projection. If for some closed point $x\in \mathcal{X}$, the eigenvalues of $\rho_{\lambda}^{ab}(\mathrm{Frob}_x)$ are plain of characteristic $p$ (for the precise definition of $G$-irreducibility and plain of characteristic $p$, see Definition \ref{definition 5.1}), then there exists an $G$-irreducible representation $\rho_{\lambda'}:\Pi\to G(\overline{\Q_{\lambda'}})$ which is a companion of $\rho_{\lambda}$ in the sense of Definition \ref{definition 3.7} and such that $\rho_{\lambda'}^{ab}$ is compatible with $\rho_{\lambda}^{ab}$.
\end{theorem}

This result is then applied to the study of compatibility of the canonical $\ell$-adic local systems on Shimura stacks. The compatibility of the canonical $\ell$-adic local systems on Shimura varieties is proved in stages: Kisin first proved it for Shimura varieties of abelian types \cite[Corollary 2.3.1]{MR3630089}, and then Klevdal and Patrikis, \cite[theorem 1.1]{klevdal2023compatibility}, proved it for Shimura varieties of non-abelian types but with adjoint value groups. 
With this result, Bakker, Shankar, and Tsimerman \cite[theorem 1.5]{bakker2025integralcanonicalmodelsexceptional} showed that all Shimura varieties have canonical integral models away from a finite set of bad primes. This result finally leads to the proof by Klevdal and Patrikis of compatibility of the canonical $\ell$-adic local systems on Shimura varieties of non-abelian types with reductive value groups \cite[theorem 1.1]{patrikis2025compatibilitycanonicalelladiclocal}. These techniques lead to the second main result of this paper:
\begin{theorem}[see Theorem \ref{theorem 6.29}]
    Let $(G,X)$ be a Shimura datum such that $Z_G(\mathbb{Q})$ is a discrete subgroup of $Z_G(\mathbb{A}_f)$, $K \subset G(\mathbb{A}_f)$ a compact open subgroup (not necessarily neat), and let $[S_{K,s}]$ be a connected component of the Shimura stack arising from the Shimura datum $(G,X)$. Assume that for all $\mathbb{Q}$-simple factors $H$ of $G^{\mathrm{ad}}$, $\operatorname{rk}_{\mathbb{R}}(H_{\mathbb{R}}) \geq 2$. Then there is an integer $N$, an integral model $\mathbb{S}_{K,s} \text{over } \mathcal{O}_{E_{K,s}}[1/N]$ for $[S_{K,s}]$ such that for all closed points $x \in \mathbb{S}_{K,s}[1/\ell]$, the class of $\rho_\ell(\mathrm{Frob}_x)$ lies in $[G//G](\Q)$ and is independent of $\ell$ (not equal to the residue characteristic of $x$).
\end{theorem}

To compare this theorem with the compatibility theorem on Shimura varieties: the compatibility statements for Shimura varieties work only after replacing the level by a neat subgroup, so that the moduli object becomes a scheme and every rational point corresponds to an object with trivial automorphism group. In contrast, the above theorem is formulated directly for Shimura stacks at arbitrary level.

A key point is that Shimura stacks typically have more rational points than any neat-level Shimura variety: stack points may involve nontrivial automorphism group schemes and twisted Galois descent data, and such points often do not descend to rational points on any neat model.
Therefore the theorem proved here applies to points that are invisible to the classical Shimura-variety theorems, making the result stronger in scope.

\section{Acknowledgements}
This work was supported by the NSF grant DMS-1752313/DMS-2120325.

\section{Background and Notations}
The companion part of Deligne's conjecture for smooth schemes is proved in 2012:

\begin{theorem}[{\cite[theorem 1.1] {drinfeld2018conjecture}}]
\label{theorem 3.1}
    Let $X$ be a smooth scheme over $\mathbb{F}_p$. Let $E$ be a finite extension of $\Q$. Let $\lambda, \lambda'$ be nonarchimedean places of $E$ not dividing $p$ and $E_{\lambda}, E_{\lambda'}$ the corresponding completions. Let $F_x$ be the geometric Frobenius for a closed point $x\in X$ and $\mathcal{E}$ be a lisse $E_{\lambda'}$-sheaf on $X$ such that for every closed point $x \in X$ the polynomial $\mathrm{det}(1 - F_xt, \mathcal{E})$ has coefficients in E and its roots are $\lambda$-adic units. Then there exists a lisse $E_{\lambda}$-sheaf on $X$ compatible with $\mathcal{E}$ (i.e., having the same characteristic polynomials of the operators $F_x$ for all closed points $x \in X$).
\end{theorem}

There are two important generalizations, proved by Zheng and Drinfeld respectively, of the above theorem.

\subsection{Zheng's generalization}
In rough terms, Zheng generalized by replacing smooth schemes with smooth algebraic stacks. Before introducing his generalization, we first introduce some definitions and notations.

The definitions of algebraic stacks and DM stacks will be as follows:
\begin{definition}[{\cite[D\'EFINITION 4.1]{LMB}}]
\label{definition 3.1}
Let $S$ be a base scheme. An algebraic $S$-stack is an $S$-stack $\mathcal{X}$ which satisfies the following axioms:
\begin{enumerate}
    \item[(i)] the diagonal $1$-morphism of $S$-stacks:
    \begin{align*}
    \mathcal{X}\xrightarrow{\triangle}\mathcal{X}\times_S \mathcal{X}
    \end{align*}
    is representable, separated and quasi-compact.
    \item[(ii)] there exists an $S$-algebraic space and an $1$-morphism of $S$-stacks $X\xrightarrow{P}\mathcal{X}$ which is surjective and smooth.
\end{enumerate}
An $1$-morphism $P$ as in $(ii)$ is called a presentation of $\mathcal{X}$. A Deligne-Mumford $S$-stack (DM $S$-stack for short) is an algebraic $S$-stack which admits an \'etale presentation.
\end{definition}

Note that the above definition of DM stacks is exactly the same as that in Deligne and Mumford's paper, "the irreducibility of the space of curves of given genus", \cite[definition 4.6]{DM}.

In this paper, algebraic spaces and algebraic stacks over a base scheme $S$ are taken to be over $(\mathrm{Sch}/S)_{\'et}$.

In Zheng's Paper, a notion of "companion" is introduced:
\begin{definition}[{\cite[Introduction]{Zheng_2018}}]
\label{definition 3.3}
Let $\mathcal{X}$ be an algebraic stack of finite presentation over $\mathbb{F}_q$ and let $\ell,\ell'$ be two primes not dividing $q$. Let $\mathcal{F}$ be a lisse $\overline{\Q_\ell}$-sheaf on $\mathcal{X}$. Let $E(\mathcal{F})$ denote the subfield of $\overline{\Q_\ell}$ generated by the local Frobenius traces $\mathrm{tr}(\mathrm{Frob}_x, \mathcal{F}_{\overline{x}})$,
where $x \in \mathcal{X}(\mathbb{F}_{q^n})$ and $n \geq 1$. Here $\mathrm{Frob}_x = \mathrm{Frob}_{q^n}$ denotes the geometric Frobenius, and $\overline{x}$ denotes a geometric point above $x$. Let $E_{\lambda'}$ be an algebraic extension of $\overline{\Q_{\ell'}}$ and let $\sigma : E(\mathcal{F})\to E_{\lambda'}$ be a field embedding, not necessarily continuous. We say that a lisse $E_{\lambda'}$-sheaf $F'$ is a $\sigma$-companion of $\mathcal{F}$ if for all $x \in X(\mathbb{F}_{q^n})$ with $n \geq 1$, we have $\mathrm{tr}(\mathrm{Frob}_x, \mathcal{F}'_{\overline{x}}) = \sigma \mathrm{tr}(\mathrm{Frob}_x, \mathcal{F}_{\overline{x}})$.
\end{definition}

And his main result is 

\begin{theorem}[{\cite[Theorem 0.1,0.2]{Zheng_2018}}]
\leavevmode
    \begin{enumerate}
        \item Let $\mathcal{X}$ be a geometrically unibranch (see \cite[Definition 0DQH]{stacks-project} for the definition) algebraic stack of finite presentation over $\mathbb{F}_{q}$ and let $\mathcal{F}$ be a simple lisse $\overline{\Q_\ell}$-sheaf of rank r on $\mathcal{X}$ such that $\mathrm{det}(\mathcal{F})$ has finite order, then the field $E(\mathcal{F})$ is a number field.
        
        \item Let $\mathcal{X}$ be a smooth algebraic stack over $\mathbb{F}_{q}$ of finite presentation. Let $\mathcal{F}$ be a lisse Weil $\overline{\Q_\ell}$-sheaf on $\mathcal{X}$. Then, for every embedding $\sigma : E(\mathcal{F}) \to \overline{\Q_{\ell'}}$, $\mathcal{F}$ admits a lisse $\sigma$-companion $\mathcal{F}'$. Moreover, if $E(\mathcal{F})$ is a number field, then there exists a finite extension $E$ of $E(\mathcal{F})$ such that for every finite place $\lambda'$ of $E$ not dividing $q$, $\mathcal{F}$ admits a lisse $\sigma_{\lambda'}$-companion. Here $\sigma_{\lambda'}: E(\mathcal{F}) \to E \to E_{\lambda'}$,and $E_{\lambda'}$ denotes the completion of E at $\lambda'$.
    \end{enumerate}
\end{theorem}

\subsection{Drinfeld's generalization}
Hereafter, semisimple algebraic groups will mean semisimple and not necessarily connected algebraic groups. Similarly, reductive algebraic groups will mean reductive and not necessarily connected algebraic groups. If the algebraic group is connected, it will be indicated explicitly.

If we regard lisse sheaves as representations of the fundamental group valued in $\mathrm{GL}_n$, then it is a natural question that whether Theorem \ref{theorem 2.1} still hold if $\mathrm{GL}_n$ is replaced by other reductive groups or semisimple groups (The finite order determinant assumption will be kept of course), and this is the direction where Drinfeld generalizes the Theorem \ref{theorem 2.1}.

We will first introduce the notations in Drinfeld's work before looking at his theorem.

Let $X$ be an irreducible normal variety over $\mathbb{F}_p$, and denote the \'etale fundamental group of $X$ by $\pi_1(X,\overline{x_0})$, for some closed point $x_0\in X$. To ease notations, denote by $\Pi$ the fundamental group $\pi_1(X,\overline{x_0})$. Let $\hat{\Pi}_{\ell}=\varprojlim(G,\rho)$, where $G$ is a semisimple algebraic group over $\Q_\ell$ and $\rho:\Pi\to G(\Q_\ell)$ is a continuous homomorphism with Zariski dense image. Fix a universal cover $(\widetilde{X},\widetilde{x_0})$ of $(X,x_0)$, denote by $F_{\widetilde{x}}$ the automorphism of $\widetilde{X}$ with respect to $X$ whose restriction to $\{\widetilde{x}\}$ is the following composition:
\begin{center}
\begin{tikzcd}
\{\widetilde{x}\}\arrow[r,"\phi"]& \{\widetilde{x}\} \arrow[hookrightarrow]{r}& \widetilde{X}
\end{tikzcd}
\end{center}
where $\phi$ is the Frobenius of $\widetilde{x}$ with respect to $x$, and the right arrow is the canonical inclusion.

Let $\Pi_{\mathrm{Fr}}=\{F_{\widetilde{x}}^n\mid \widetilde{x}\in \widetilde{X} \text{ a closed point}, n\in \mathbb{N}\}$ be the subset of $\Pi$ of all Frobenii of closed points and their powers. And let $\lambda$ be a place of $\overline{\Q}$ over $l$, and $\hat{\Pi}_{\lambda}:=\hat{\Pi}_{\ell}\bigotimes_{\Q_\ell}\overline{\Q_{\lambda}}$. Let $\mathbf{Pro\text{-}ss}(\overline{\Q}_\lambda)$ be the groupoid consisting of pro-semisimple group scheme (projective limit of semisimple groups) over $\overline{\Q}_\lambda$, with morphisms being isomorphisms of pro-semisimple group schemes $G_1\to G_2$ up to composing with inner automorphisms of $G_2$ determined by elements of the identity component $G_2^{\circ}$ (i.e. morphisms are essentially inner isomorphisms modulo conjugation by elements of the identity component). Similarly, let $\mathbf{Pro\text{-}ss}(\overline{\Q})$ be the groupoid consisting of pro-semisimple group scheme over $\overline{\Q}$ and with similar morphisms. Finally, fix an embedding $\overline{\Q}\hookrightarrow\overline{\Q}_{\lambda}$. This induces an equivalence $F: \mathbf{Pro\text{-}ss}(\overline{\Q})\xrightarrow{\sim} \mathbf{Pro\text{-}ss}(\overline{\Q}_\lambda)$, \cite[Proposition 2.2.5]{drinfeld2018prosemisimple}. Fix a quasi-inverse $F'$ of $F$, and denote $\hat{\Pi}_{(\lambda)}:= F'(\hat{\Pi}_{\lambda})$.

We have a composition of maps: $\Pi_{\mathrm{Fr}} \to \hat{\Pi}_{\ell}(\Q_\ell)\to \hat{\Pi}_{\ell}(\overline{\Q_{\lambda}})=\hat{\Pi}_{\lambda}(\overline{\Q_{\lambda}})\cong \hat{\Pi}_{(\lambda)}(\overline{\Q_{\lambda}})\to [\hat{\Pi}_{(\lambda)}](\overline{\Q_{\lambda}})$. The first map comes from the universal property of projective limits, the second is the canonical inclusion, and the third is the projection to the GIT quotient. $\hat{\Pi}_{\lambda}(\overline{\Q_{\lambda}})\cong \hat{\Pi}_{(\lambda)}(\overline{\Q_{\lambda}})$ since $\hat{\Pi}_{\lambda}$ is isomorphic to the base change of $\hat{\Pi}_{(\lambda)}$ to $\overline{\Q_{\lambda}}$, but this isomorphism is only defined up to an inner isomorphism induced by elements of $\hat{\Pi}_{(\lambda)}^{\circ}$. But this inner automorphism is forgotten when passing to the GIT quotient, so the map $\hat{\Pi}_{\lambda}(\overline{\Q_{\lambda}})\to [\hat{\Pi}_{(\lambda)}](\overline{\Q_{\lambda}})$ is unique up to a unique isomorphism. Since $X$ is a normal variety, and the determinant of $\rho$ is of finite order, by Lafforgue, \cite[Théorème VII.7]{Lafforgue}, the image of $\Pi_{\mathrm{Fr}}$ is contained in $[\hat{\Pi}_{(\lambda)}](\overline{\Q})$.

Also, since $\pi_0(\hat{\Pi}_{(\lambda)})=\pi_0(\hat{\Pi}_{\lambda})=\pi_0(\hat{\Pi}_{\ell})=\Pi$, we have a surjective map $[\hat{\Pi}_{(\lambda)}](\overline{\Q})\to \Pi$.

This gives a diagram of sets: 
\begin{align}
\label{diag 1}
    \Pi_{\mathrm{Fr}}\to [\hat{\Pi}_{(\lambda)}](\overline{\Q})\to \Pi
\end{align}
\begin{theorem}
    [{{\cite[Theorem 1.4.1]{drinfeld2018prosemisimple}}}]
Assume that $X$ is smooth. Let $\lambda$ and $\lambda'$ be non-Archimedean places of $\overline{\Q}$ not dividing $p$. Then there exists a unique isomorphism $\Psi:\hat{\Pi}_{(\lambda)}\xrightarrow{\sim} \hat{\Pi}_{(\lambda')}$ in the category $\mathbf{Pro\text{-}ss}(\overline{\Q})$ which sends the diagram $\Pi_{\mathrm{Fr}} \to [\hat{\Pi}_{(\lambda)}](\overline{\Q})\to \Pi$ to a similar diagram $\Pi_{\mathrm{Fr}} \to [\hat{\Pi}_{(\lambda')}](\overline{\Q})\to \Pi$. 
\end{theorem}

\begin{remark}
    Let $\rho$ be a semisimple representation $\rho: \Pi\to \mathrm{GL}_n(\overline{\Q_{\lambda}})$ such that every irreducible component of $\rho$ has a finite order determinant. This can be identified with a semisimple lisse $\overline{\Q_{\lambda}}$-sheaf $\mathcal{E}$ on $X$ such that the determinant of each irreducible components has finite order. Call the category of such sheaves as $T_{\lambda}(X)$, where the morphism are the obvious ones. Drinfeld also defined companion as follows:
    \begin{definition}
    \label{definition 3.7}
    $\mathcal{E}\in T_{\lambda}(X)$ is said to have companion if for every non-Archimedean place $\lambda'$ of $\overline{\Q}$ coprime to $p$, there exists $\mathcal{E}'\in T_{\lambda'}$ such that $\mathrm{det}(1 - t\cdot F_x, \mathcal{E}) = \mathrm{det}(1 - t\cdot F_x, \mathcal{E}')$ for all $x \in |X|$.
    \end{definition}
\end{remark}
The above definition coincides with Definition \ref{definition 3.3} by Newton's identity on symmetric polynomials.

\begin{remark}[Interpretation of Drinfeld's theorem]
\label{remark 3.8}
Let $G$ be a semisimple algebraic $\Q$-group. Given an irreducible $\rho_{\lambda}:\Pi\to G(\overline{\Q_{\lambda}})$, it is the same as a representation of $\hat{\Pi}_{\lambda}$ into $\mathrm{GL}_{n,\overline{\Q_{\lambda}}}$. By the equivalence of categories: $\mathbf{Pro\text{-}ss}(\overline{\Q})\xrightarrow{\sim} \mathbf{Pro\text{-}ss} (\overline{\Q}_\lambda)$, we obtain a representation of $\hat{\Pi}_{(\lambda)}$ into $\mathrm{GL}_{n,\overline{\Q}}$. $\Psi^{-1}$, the inverse of the unique isomorphism, induces a representation of $\hat{\Pi}_{(\lambda')}$ into $\mathrm{GL}_{n,\overline{\Q}}$. Again by the equivalence of categories: $\mathbf{Pro\text{-}ss}(\overline{\Q})\xrightarrow{\sim} \mathbf{Pro\text{-}ss}(\overline{\Q}_{\lambda'})$, we obtain a representation of $\rho_{\lambda'}:\hat{\Pi}_{\lambda'}$ into $\mathrm{GL}_{n,\overline{\Q_{\lambda'}}}$. Since the $\Psi$ respect the Diagram \ref{diag 1} for both $\lambda$ and $\lambda'$, $\rho_{\lambda}$ and $\rho_{\lambda'}$ are compatible (i.e., they have the same Frobenius characteristic polynomials at all closed points).
\end{remark}

\begin{remark}[Brief explanation of Drinfeld's idea]
    The difficulty for this theorem lies in the existence of the isomorphism. Before discussing this, let's first recall the definition of the lambda-semiring in \cite{drinfeld2018prosemisimple}.
    \end{remark}
    \begin{definition}
        Let $A$ be a ring. A lambda-semiring is a pair consisting of a cancellation semiring $A^+$ and a lambda-structure on the corresponding ring $A$ such that $\lambda^i(A^+) \subseteq A^+$ for all $i \in \mathbb{N}$.
    \end{definition}
    For the group $\hat{\Pi}_{(\lambda)}$, we define the lambda-semiring $K^{+}(\hat{\Pi}_{(\lambda)})$ as the pair consisting of the Grothendieck semiring and the lambda structure $\lambda^i(\rho)=\wedge^i\rho$.
    The idea to construct the above isomorphism is to use the lambda-semiring $K^{+}(\hat{\Pi}_{(\lambda)})$ to reconstruct $\hat{\Pi}_{(\lambda)}$. There is a theorem by Kazhdan-Larsen-Varshavsky:
\begin{theorem}[{{\cite[Theorem 1.2]{Kazhdan}}}]
    Let $G,G'$ be connected pro-reductive groups over an algebraically closed field $E$ of characteristic $0$. Then any semiring isomorphism $\phi:K^+(G)\xrightarrow{\sim}K^+(G')$ is induced by an isomorphism $f:G\cong G'$, which is unique up to conjugation by elements of $G'(E)$.
\end{theorem}
This can be interpreted as the full faithfulness of the functor from the subcategory of $\mathbf{Pro\text{-}ss}(E)$ consisting of connected group schemes to the category of Grothendieck semirings. But when $G$ is not necessarily connected, the structure of Grothendieck semiring is not enough, and we need to consider the exterior product operation on the Grothendieck semirings, and also those semirings for $G\times_{\Pi}U$, where $U$ ranges from all open subgroups of $\Pi$. There is a functor recording those information, and it restricts to a certain subcategory to be a fully faithful functor, where the subcategory is defined by some cohomological conditions.

Drinfeld first proved that for affine curves $X$, the cohomological conditions hold. After that, the theorem was extended to all curves, essentially by the fact that for open immersion $U\hookrightarrow X$, $\pi_1(U)\to \pi_1(X)$ is surjective, and finally to all varieties by Hilbert irreducibility theorem.

\section{Main theorem for semisimple $G$}
We first recall the notion of the fundamental group of an algebraic stack.

\begin{definition}[{\cite[\textsection 4]{noohi2002fundamental}}]
The fundamental group of an algebraic stack $\mathcal{X}$, with the choice of a certain geometric point, is defined as the automorphism group of the fiber functor of some Galois category. The Galois category is given as follows:
\begin{enumerate}
    \item The objects of the Galois category are pairs $(\mathcal{Y},f)$, where $\mathcal{Y}$ is an algebraic stack and $f:\mathcal{Y}\to \mathcal{X}$ is a covering space.
    \item For any two objects $(\mathcal{Y},f)$ and $(\mathcal{Z},g)$, the morphisms between them are defined as 
    \begin{align*}
        \mathrm{Mor}((\mathcal{Y},f),(\mathcal{Z},g))=\{(a,\Phi) | a:\mathcal{Y}\to \mathcal{Z} \text{ a morphism of algebraic stacks},\\
        \Phi: f\Rightarrow g\circ a \text{ a 2-morphism }\}_{/\sim},
    \end{align*}
    where $\sim$ is defined by $(a,\Phi)\sim (b,\Psi)$ if there is a $\Gamma:a\to b$ such that $g(\Gamma)\circ \Phi=\Psi$.
\end{enumerate}
\end{definition}

Let $\mathcal{X}$ be a connected normal algebraic stack of finite presentation over $\mathbb{F}_q$, and let $x$ be a geometric point underlying a closed point of $\mathcal{X}$. Denote again by $\Pi$ the fundamental group of $\mathcal{X}$ with base point $x$. Let $\hat{\Pi}_{\ell}=\varprojlim(G,\rho)$, where $G$ is a semisimple algebraic group over $\Q_\ell$ and $\rho:\Pi\to G(\Q_\ell)$ is a continuous homomorphism with Zariski dense image. We can similarly fix a universal cover $\widetilde{\mathcal{X}}$ as the projective limit of all the finite \'etale covers of $\mathcal{X}$ and define: $\Pi_{\mathrm{Fr}}=\{F_{\widetilde{x}}^n\mid \widetilde{x}\in \widetilde{\mathcal{X}} \text{ a closed point}, n\in \mathbb{N}\}$.

The composition of maps $\Pi_{\mathrm{Fr}} \to \hat{\Pi}_{\ell}(\Q_\ell)\to \hat{\Pi}_{\ell}(\overline{\Q_{\lambda}})=\hat{\Pi}_{\lambda}(\overline{\Q_{\lambda}})\cong \hat{\Pi}_{(\lambda)}(\overline{\Q_{\lambda}})\to [\hat{\Pi}_{(\lambda)}](\overline{\Q_{\lambda}})$ can be obtained in a similar way. Finally, since $\mathcal{X}$ is a connected normal algebraic stack of finite presentation over $\mathbb{F}_q$, 
\cite[theorem 0.1]{Zheng_2018} shows that the image of $\Pi_{\mathrm{Fr}}$ is contained in $[\hat{\Pi}_{(\lambda)}](\overline{\Q})$. (by scheme theoretic reasons, since if $f:X\to Y$ is a $\Q$-morphism of schemes, then the preimage of a closed $\overline{\Q}$-point is also defined over $\overline{\Q}$).

And then we obtain a similar diagram of sets: 
\begin{align*}
    \Pi_{\mathrm{Fr}}\to [\hat{\Pi}_{(\lambda)}](\overline{\Q})\to \Pi
\end{align*}
\begin{theorem}
\label{theorem 4.2}
Let $\mathcal{X}$ be a smooth algebraic stack of finite presentation over $\mathbb{F}_q$. Let $\lambda$ and $\lambda'$ be non-Archimedean places of $\overline{\Q}$ not dividing $p$. Then there exists a unique isomorphism $\phi:\hat{\Pi}_{(\lambda)}\xrightarrow{\sim} \hat{\Pi}_{(\lambda')}$ in the category $\mathbf{Pro\text{-}ss}(\overline{\Q})$ which makes the following diagram commutative:
\begin{equation}
\label{diag 2}
    \begin{tikzcd}
        \Pi_{\mathrm{Fr}}\arrow[r,""]\arrow[d,"\mathrm{id}"] &{[\hat{\Pi}_{(\lambda)}](\overline{\Q})}\arrow[r,""]\arrow[d,"\phi","\rotatebox{90}{$\sim$}"'] &\Pi\arrow[d,"\mathrm{id}"]\\
        \Pi_{\mathrm{Fr}}\arrow[r,""] &{[\hat{\Pi}_{(\lambda')}](\overline{\Q})}\arrow[r,""] &\Pi
    \end{tikzcd}
\end{equation}
\end{theorem}

\begin{proposition}
\label{proposition 4.3}
    The isomorphism $\phi$ is unique.
\end{proposition}
\sloppy
\begin{proof}
    We follow Drinfeld's ideas here. Drinfeld uses a series of lemmas to establish the proposition. We analyze them one by one.
        
    Observe that to prove this lemma, Drinfeld uses the fact that for a finite index subgroup $H$ of $G$, an $H$-module $V$ injects into its induction $Ind_G^H V$ equivariantly. So this lemma is purely group theoretical, and it also holds for the prosemisimple completion of $\Pi$ (with $\Pi$ being the fundamental group of a scheme or an algebraic stack satisfying the condition of Theorem \ref{theorem 4.2}).

    The second lemma is \cite[lemma 3.1.3]{drinfeld2018prosemisimple}. What are used here are the following: first, for any representation $\rho:\Pi\to G(\Q_\ell)$, where $G$ is an algebraic group, $\rho$ is $\Pi$-equivariant; second, the universal property of projective limit and third the \cite[lemma 3.1.2]{drinfeld2018prosemisimple}.

    The third lemma in order is \cite[proposition 3.1.4]{drinfeld2018prosemisimple}. There, the following fact is used:  a connected group scheme is simply connected is equivalent to for any étale homomorphism \( \varphi : H' \to H \) of algebraic groups over \(\mathbb{Q}_\ell \), the corresponding map \( \mathrm{Hom}(\widetilde{\Pi}_\ell^\circ, H') \to \mathrm{Hom}(\widetilde{\Pi}_\ell^\circ, H) \) is bijective. The reason is that the scheme theoretical fundamental group of a connected group scheme is profinite, and that the latter condition on the above is equivalent to the fact that the profinite completion is trivial.

    Another fact is also used: for an etale homorphism of \( \varphi : H' \to H \) over $\Q_\ell$, there exists an open subgroup \( U \subset H'(\mathbb{Q}_\ell) \) such that the map \( U \to H(\mathbb{Q}_\ell) \) induced by \( \varphi \) is an open embedding. This is because \'etale homomorphism is a local isomorphism.

    The fourth lemma is \cite[proposition 3.2.2]{drinfeld2018prosemisimple}. Drinfeld proved it in a way which seems to depend on the fact that $X$ is a scheme. But we can use a standard Baire category theorem argument to prove it for all profinite group $\Pi$.
    
    \begin{Sublemma}
    \label{Sublemma 4.4}
        \textit{Let \( E \subset \overline{\mathbb{Q}}_\lambda \) be a subfield finite over \( \mathbb{Q}_\ell \), \( H \) an affine algebraic group over \( E \), and \( f : \Pi \to H(\overline{\mathbb{Q}}_\lambda) \) a continuous homomorphism. Then there is a field \( E' \subset \overline{\mathbb{Q}}_\lambda \) finite over \( E \) such that \( f(\Pi) \subset H(E') \).}
    \end{Sublemma}
    \begin{proof}
    Since $\Pi$ is a compact topological group, $f(\Pi)$ is also a compact topological group. Since $f(\Pi)=\bigcup_{E/\Q_\ell \text{ finite }} f(\Pi)\cap H(E)$ is a countable union of closed subsets, Baire category theorem implies that there must be some $E$ such that $f(\Pi)\cap H(E)$ contains an open subset of $f(\Pi)$. Using the compactness of $f(\Pi)$ again, we conclude that there is an finite extension $E'/E$ such that $f(\Pi)\subseteq H(E')$. Since $E'/E$ is a finite extension, $E'/\Q_\ell$ is also a finite extension.
    \end{proof}
    The last lemma is \cite[proposition 2.4.4]{drinfeld2018prosemisimple}, which is purely algebraic-group-theoretical.

    Therefore, it is clear that the above proposition in the scheme case is algebraic-group-theoretical and thus extends to the stack case.

\end{proof}

\subsection{Reduce to dense open subschemes}
To do this, we need the following lemma.
\begin{lemma}
\label{lemma 5.5}
    A normal connected quasi-compact locally Noetherian algebraic stack $\mathcal{X}$ over a field is irreducible.
\end{lemma}
\begin{proof}
First assume that $\mathcal{X}$ is a DM stack. By \cite[theorem 5.44]{DM_stack}, $\mathcal{X}$ has an \'etale presentation $X\to \mathcal{X}$ such that $X$ is connected. $X$ is also normal as being normal is a property local in the \'etale topology. Therefore, $X$, as a normal connected scheme, is irreducible. Consequently, $\mathcal{X}$, whose topological space is a continuous image of an irreducible space, is also irreducible.

When $\mathcal{X}$ is more generally an algebraic stack, by \cite[Proposition 5.1.11, 5.1.14]{Behrend}, there is an open dense substack $\mathcal{U}$ of $\mathcal{X}$, and a gerbe-like morphism $g:\mathcal{U}\to \mathcal{Y}$, where $\mathcal{Y}$ is a DM stack. Since being normal is smooth-local, $\mathcal{U}$ is also normal.

Let me prove that $\mathcal{U}$ is connected. Suppose not. 
Choose a smooth presentation $X\to \mathcal{X}$. One has the following Cartesian diagram:
\begin{center}
    \begin{tikzcd}
    \mathcal{U}\times_\mathcal{X} X\arrow[r,"f'"]\arrow[d,]& X \arrow[d,]\\
    \mathcal{U}\arrow[hookrightarrow]{r}{f} &\mathcal{X}
    \end{tikzcd}
\end{center}
Denote $\mathcal{U}\times_\mathcal{X} X$ by $U$. Since $\mathcal{U}\to \mathcal{X}$ is an open immersion, it is representable by schemes, so $U$ is an open subscheme of $X$. Since $\mathcal{X}$ is quasi-compact and locally Noetherian, $X$ is so as well. So being an open subscheme of a Noetherian scheme $X$, $U$ is quasi-compact.

Suppose $\mathcal{U}$ is not connected, we can write $\mathcal{U}=\amalg_{i=1}^n\mathcal{U}_i$, where $\mathcal{U}_i$ are open substacks of $\mathcal{U}$. Suppose for some $i,j$, $\overline{\mathcal{U}_i}\cap \overline{\mathcal{U}_j}$ is non-empty, where $\overline{\mathcal{U}_i}$ denotes the topological closure of $\overline{\mathcal{U}_i}$ in $\mathcal{X}$. Let $u\in \overline{\mathcal{U}_i}\cap \overline{\mathcal{U}_j}$, $X'\to \mathcal{X}$ a representable smooth morphism from a connected scheme $X'$ with $u$ in the image of $X'$. As a normal connected scheme, $X'$ is irreducible. However, $X'\times_{\mathcal{X}}\mathcal{U}_i$ and $X'\times_{\mathcal{X}}\mathcal{U}_j$ are disjoint nonempty open subschemes of $X'$, contradicting the fact that $X'$ is irreducible. Thus, for any pair $(i,j)$, $\overline{\mathcal{U}_i}$ is disjoint from $\overline{\mathcal{U}_j}$.

Since $\mathcal{U}$ is an open dense substack of $\mathcal{X}$, $\amalg_{i=1}^n\overline{\mathcal{U}_i}=\mathcal{X}$, this is a contradiction as $\mathcal{X}$ is assumed to be connected. Therefore $\mathcal{U}$ is connected. Since $g:\mathcal{U}\to \mathcal{Y}$ is a gerbe-like morphism, $g$ is a smooth universal homeomorphism. Let $U''\to \mathcal{U}$ be a smooth presentation, since $g$ is a smooth universal homeomorphism, $U''\to \mathcal{Y}$ is also a smooth presentation. Since $\mathcal{U}$ is normal, $U''$ is also normal, so $\mathcal{Y}$ is normal as well. Since $\mathcal{Y}$ is a normal connected DM stack, it is irreducible, so $\mathcal{U}$ is irreducible as well. Since $\mathcal{U}$ is open and dense in $\mathcal{X}$ and is irreducible, $\mathcal{X}$ is also irreducible.
\end{proof}
\begin{remark}
\label{remark 4.6}
    Combined with \cite[Lemma 0DR5]{stacks-project}, this lemma also proves that a normal connected algebraic stack over a field has a smooth presentation by a connected scheme.
\end{remark}
\begin{lemma}
\label{lemma 4.7}
If $\mathcal{U}$ is an open substack of a normal irreducible quasi-compact locally Noetherian algebraic stack $\mathcal{X}$, then $\pi_1(\mathcal{U})\to \Pi$ is surjective.
\end{lemma}
\begin{proof}
Let $f:\mathcal{Y}\to \mathcal{X}$ be a connected finite \'etale cover of $\mathcal{X}$, and let $X\to \mathcal{X}$ be a smooth presentation of $\mathcal{X}$.
There is a pull-back diagram:
    \begin{center}
    \begin{tikzcd}
    \mathcal{Y}\times_\mathcal{X} X\arrow[r,"f'"]\arrow[d,]& X \arrow[d,]\\
    \mathcal{Y}\arrow[r,"f"]&\mathcal{X}
    \end{tikzcd}
    \end{center}
Denote $\mathcal{Y}\times_\mathcal{X} X$ by $Y$. Since $f$ is representable by algebraic spaces, and $X$ is a scheme, $Y$ is an algebraic space. Since $f$ is a finite \'etale map, $f'$ is also finite \'etale. Similarly, $Y\to \mathcal{Y}$ is also a smooth presentation. Since $Y\to X$ is finite and \'etale, by the criterion that a locally quasi-finite and separated algebraic space over a scheme is a scheme (see \cite{Moduli}), $Y$ is a scheme. Since $X$ is normal and that $f'$ is finite \'etale, $Y$ is normal. Thus, $\mathcal{Y}$ is normal as well. Since $\mathcal{Y}$ inherits other properties from $\mathcal{X}$ as well (quasi-compactness, local-Noetherianess, connectedness), by Lemma \ref{lemma 5.5}, $\mathcal{Y}$ is irreducible.

If $\mathcal{U}$ is an open substack of $\mathcal{X}$, then $\mathcal{U}\times_\mathcal{X}\mathcal{Y}$ is an open substack of the irreducible algebraic stack $\mathcal{Y}$, so $\mathcal{U}\times_\mathcal{X}\mathcal{Y}$ is irreducible and thus connected.

Therefore, by Galois formalism, this shows that the induced map $\pi_1(\mathcal{U})\to \Pi$ is surjective.
\end{proof}

Let $K^{+}(\hat{\Pi}_{\lambda})$ be the lambda-semiring of $\mathrm{Rep}(\hat{\Pi}_{\lambda})$. By the definition of $\hat{\Pi}$, $K^{+}(\hat{\Pi}_{\lambda})$ is the same as the Grothendieck semiring of $T_{\lambda}(X)$. Because $\mathbf{Pro\text{-}ss}(\overline{\Q})$ is equivalent to $\mathbf{Pro\text{-}ss}(\overline{\Q_{\lambda}})$, $K^{+}(\hat{\Pi}_{\lambda})=K^{+}(\hat{\Pi}_{(\lambda)})$. Similarly, we have $K^{+}(\hat{\Pi}_{\lambda'})=K^{+}(\hat{\Pi}_{(\lambda')})$.

We can also realize $K^{+}(\hat{\Pi}_{(\lambda)})$ as a lambda-subsemiring of the algebra of functions $\Pi_{\mathrm{Fr}}\to \overline{\Q_{\lambda}}$ invariant under $\Pi$-conjugation by mapping every representation to its trace character. This map is injective by Brauer–Nesbitt theorem and Chebotarev's density theorem.

We are in a similar situation in the stacky case.

\begin{lemma}
\label{lemma 4.8}
    \begin{enumerate}
        \item[(i)] $K^+(\hat{\Pi}_{(\lambda)})$ is a lambda-subsemiring of the algebra of functions $\Pi_{\mathrm{Fr}}\to \overline{\Q}$.
        \item[(ii)] $K^+(\hat{\Pi}_{(\lambda)})$ is independent of $\lambda$.
    \end{enumerate}
\end{lemma}
\begin{proof}
Denote by $A$ the algebra of functions $\Pi_{\mathrm{Fr}}\to \overline{\Q}$. For $(i)$, let $\omega$ be a map from $K^+(\hat{\Pi}_{\lambda})$ to $A$ by mapping the representations to their trace characters. \cite[Proposition 4.6(i)]{Zheng_2018} shows that $\omega$ is injective, so we can regard $K^+(\hat{\Pi}_{(\lambda)})=K^+(\hat{\Pi}_{\lambda})$ as a lambda-subsemiring of $A$ with lambda operation being $(\lambda^if)(g)=f(g^i)$.

For $(ii)$, for an element in $K^{+}(\hat{\Pi}_{\lambda})$, which we can regard as a semisimple representation $\rho:\Pi\to \mathrm{GL}_n(\overline{\Q_{\lambda}})$ with each irreducible component having finite order determinant. By \cite[theorem 0.2]{Zheng_2018}, each irreducible component of $\rho$ has a $\lambda'$-adic companion. The direct sum of them is semisimple and is a companion of $\rho$ such that every irreducible component has a finite order determinant. Thus, we get a map $\alpha_1$ from $K^{+}(\hat{\Pi}_{\lambda})$ to $K^{+}(\hat{\Pi}_{\lambda'})$. Similarly, we can also get a map $\alpha_2$ from $K^{+}(\hat{\Pi}_{\lambda'})$ to $K^{+}(\hat{\Pi}_{\lambda})$. Since companions have the same trace character and that both $K^{+}(\hat{\Pi}_{\lambda})$ and $K^{+}(\hat{\Pi}_{\lambda'})$ embeds into $A$, $\alpha_1$ and $\alpha_2$ are inverse to each other. By the equality $K^+(\hat{\Pi}_{(\lambda)})=K^+(\hat{\Pi}_{\lambda})$ and $K^+(\hat{\Pi}_{(\lambda')})=K^+(\hat{\Pi}_{\lambda'})$, we can identify $K^{+}(\hat{\Pi}_{(\lambda)})$ and $K^{+}(\hat{\Pi}_{(\lambda')})$. 
\end{proof}

\begin{lemma}
\label{lemma 4.9}
    If $\mathcal{U}$ is an open substack of a smooth irreducible algebraic stack $\mathcal{X}$, and Theorem \ref{theorem 4.2} holds for $\mathcal{U}$, then it holds for $\mathcal{X}$.
\end{lemma}
\begin{proof}
We prove the lemma in the following steps:

Denote $\pi_1(\mathcal{U})$ by $\Pi^1$. Choose an isomorphism $\Psi: \Pi^1_{(\lambda)}\xrightarrow{\sim} \Pi^1_{(\lambda')}$. By Remark \ref{remark 3.8}, the isomorphism $\hat{\Pi}^1_{(\lambda)}\cong \hat{\Pi}^1_{(\lambda')}$ induces the isomorphism of corresponding lambda semirings. By Lemma \ref{lemma 4.7}, $\Pi$ is a quotient of $\Pi^1$. So $K^+(\hat{\Pi}_{(\lambda)})$ and $K^+(\hat{\Pi}_{(\lambda')})$ are lambda subsemirings of $K^+(\Pi^1_{(\lambda)})=K^+(\Pi^1_{(\lambda')})$ respectively. By the identification in Lemma \ref{lemma 4.8}, the following diagram commutes, where $A$ is the algebra of functions from the set of Frobenius elements of $\Pi^1$ to $\overline{\Q}$: 

\begin{equation}\label{diag 3}
\begin{tikzcd}
    K^+(\hat{\Pi}_{(\lambda')}) \arrow[dd, "\rotatebox{90}{$\sim$}"] \arrow[hookrightarrow]{r} 
        & K^+(\hat{\Pi}^1_{(\lambda')}) \arrow[dd, "\rotatebox{90}{$\sim$}"] \arrow[hookrightarrow]{dr} \\
    && A \\
    K^+(\hat{\Pi}_{(\lambda)}) \arrow[hookrightarrow]{r} 
        & K^+(\hat{\Pi}^1_{(\lambda)}) \arrow[hookrightarrow]{ur}
\end{tikzcd}
\end{equation}

Thus, the isomorphism $\Psi$ also induces an isomorphism $\hat{\Pi}_{(\lambda)}\xrightarrow{\sim}\hat{\Pi}_{(\lambda')}$, for otherwise the left vertical arrow of the above graph will not exist. Indeed, suppose that the kernel for the quotient map $\Pi^{1}_{(\lambda)}\twoheadrightarrow \Pi_{(\lambda)}$ is $K_{\lambda}$ and its counterpart in the $\lambda'$ case is $K_{\lambda'}$. It suffices to prove that $\Psi(K_{\lambda})\subseteq K_{\lambda'}$, since we can do the same for $\Psi^{-1}$. Suppose the opposite, that is, $\Psi(K_{\lambda})\not\subseteq K_{\lambda'}$. Then by the projective limit structure of $\Pi^{1}_{(\lambda')}$, both $K_{\lambda'}$ and $\Psi(K_{\lambda})$, as closed subgroups of $\Pi^{1}_{(\lambda')}$, are also projective limits. Since $\Psi(K_{\lambda})\not\subseteq K_{\lambda'}$, there exists an algebraic group $H$ and two closed subgroups $K_1,K_2$ of $H$, such that where $H,K_1,K_2$ are quotient of $\Pi^{1}_{(\lambda')},\Psi(K_{\lambda}), K_{\lambda'}$ respectively and that $K_1\not\subseteq K_2$. Let $\rho_1$ be a representation of $H$ which is trivial on $K_2$ but nontrivial on $\Psi(K_1)$, and let $\rho$ be a representation composed of the projection $\Pi^{1}_{(\lambda)}\to H$ with $\rho_1$. Then $\rho$ is nontrivial on $\Psi(K_{\lambda})$ but trivial on $K_{\lambda'}$. Therefore, the image of $\rho$ under the map $K^+(\Pi^1_{(\lambda')})\to K^+(\Pi^1_{(\lambda)})$ is not trivial on all of $K_{\lambda}$, which contradicts the validity of the left vertical map in the above diagram. 

Denote the isomorphism $\hat{\Pi}_{(\lambda)}\xrightarrow{\sim} \hat{\Pi}_{(\lambda')}$ as $\Psi$ as well. Given a representation $\rho_{\lambda}:\Pi\to G(\overline{\Q_{\lambda}})$, $\Psi$ carries it to another representation $\rho_{\lambda'}:\Pi\to G(\overline{\Q_{\lambda'}})$. Let $\theta:G\to \mathrm{GL}_n$ be a faithful representation of $G$ over $\Q$, and we regard $\rho_{\lambda}$ and $\rho_{\lambda'}$ as representations into $\mathrm{GL}_n(\overline{\Q_{\lambda}})$ and $\mathrm{GL}_n(\overline{\Q_{\lambda'}})$ respectively. By the commutativity of the diagram \eqref{diag 3}, the trace character of $\rho_{\lambda}$ and $\rho_{\lambda'}$ agrees on the Frobenius elements of $\Pi^1$, which means that the pullbacks of $\rho_{\lambda}$ and $\rho_{\lambda'}$ to $\Pi^1$ are compatible. By \cite[Proposition 3.5]{Zheng_2018}, $\rho_{\lambda}$ and $\rho_{\lambda'}$ are compatible as representations of $\Pi$ as well, by the equivalence of Definition \ref{definition 3.3} and Definition \ref{definition 3.7} when $G=\mathrm{GL}_n$. This shows that the induced isomorphism makes the left half of Diagram \ref{diag 2} commute.

By the construction of $\hat{\Pi}_{\lambda}$ and $\hat{\Pi}^1_{\lambda}$ the following diagram commutes for all $n$, where the horizontal maps are composition with the projections $\hat{\Pi}_{\lambda}\to \Pi$ and $\hat{\Pi}^1_{\lambda}\to \Pi^1$ respectively and the vertical map are compositions with the quotient map:

\begin{center}
    \begin{tikzcd}
    \mathrm{Hom}(\Pi,\mathrm{GL}_{n,\overline{\Q_{\lambda}}})\arrow[d,]\arrow[r,]& \mathrm{Hom}(\hat{\Pi}_{\lambda},\mathrm{GL}_{n,\overline{\Q_{\lambda}}})\arrow[d,]\\
    \mathrm{Hom}(\Pi^1,\mathrm{GL}_{n,\overline{\Q_{\lambda}}})\arrow[r,]&\mathrm{Hom}(\hat{\Pi}^1_{\lambda},\mathrm{GL}_{n,\overline{\Q_{\lambda}}})
    \end{tikzcd}
\end{center}

Since every semisimple algebraic group is linear, the above diagram also commutes if we replace $\mathrm{GL}_n$ by other semisimple groups $G$.

Since the objects of the category $\mathbf{Pro\text{-}ss}(\overline{\Q})$ are the projective limits of semisimple groups, and by the universal property of projective limits, the above diagram also commutes if we replace semisimple algebraic group $G$ by the projective limits of semisimple algebraic groups.

Thus, the following diagram commutes by Yoneda's lemma:

\begin{center}
    \begin{tikzcd}
    \Pi& \hat{\Pi}_{\lambda}\arrow[l,]\\
    \Pi^1\arrow[u,]&\hat{\Pi}^1_{\lambda}\arrow[u,]\arrow[l,]
    \end{tikzcd}
\end{center}

Since the equivalence of categories $\mathbf{Pro\text{-}ss}(\overline{\Q})\xrightarrow{\sim} \mathbf{Pro\text{-}ss}(\overline{\Q}_\lambda)$ preserves the group of connected components, the above diagram is still commutative if we replace $\hat{\Pi}^1_{\lambda}$ by $\hat{\Pi}^1_{(\lambda)}$ and $\hat{\Pi}_{\lambda}$ by $\hat{\Pi}_{(\lambda)}$.
The situation for $\lambda'$ is the same if we replace $\lambda$ above by $\lambda'$.

Thus, combined with the fact that $\Psi^1$ respects the diagrams, the isomorphism induced by $\Psi$ makes the following diagram commutes:

\begin{center}
    \begin{tikzcd}
    \Pi& \hat{\Pi}_{(\lambda')}\arrow[l,]\\
    \Pi\arrow[u,"\mathrm{id}"]&\hat{\Pi}_{(\lambda)}\arrow[u,"\rotatebox{270}{$\sim$}"]\arrow[l,],
    \end{tikzcd}
\end{center}
i.e, the right half of Diagram \ref{diag 2} commutes with $\Psi$.

\end{proof}
\subsection{Reduction from DM stack to schemes}
\label{subsection 4.2}
\begin{theorem}
\label{theorem 4.10}
Theorem \ref{theorem 4.2} holds for a smooth global quotient DM stack $[Y/G]$ of finite presentation over $\mathbb{F}_q$, where $Y$ is affine and smooth and $G$ is a finite group scheme. Therefore, Theorem \ref{theorem 4.2} holds for all smooth connected DM stacks $\mathcal{X}$ of finite presentation over $\mathbb{F}_q$.
\end{theorem}
\begin{proof}
Let $G\to \mathrm{GL}_n$ be a faithful representation of $G$. The embedding $\mathrm{GL}_n\to \A^{n^2}$ is $\mathrm{GL}_n$-equivariant, where $\mathrm{GL}_n$ acts by left multiplication on both $\mathrm{GL}_n$ and $\A^{n^2}$. Since both the projection $Y\times \A^{n^2}\to Y$ and the zero section $Y\to Y\times \A^{n^2}$ are $G$-equivariant, they induces two morphisms of DM stacks:
\begin{align*}
    &p:[Y\times \A^{n^2}/G]\to [Y/G]\\
    &s:[Y/G]\to [Y\times \A^{n^2}/G],
\end{align*}
and $s$ is a section of $p$.

Since the action of $G$ on $Y\times \mathrm{GL}_n$ is free and that $Y\times \mathrm{GL}_n$ is quasi-projective (it is smooth), $[Y\times \mathrm{GL}_n/G]$ is a scheme. Note that $[Y\times \mathrm{GL}_n/G]$ is also open and dense in $[Y\times \A^{n^2}/G]$. Thus, by Lemma \ref{lemma 4.9}, Theorem \ref{theorem 4.2} holds for $[Y\times \A^{n^2}/G]$. Denote the fundamental group of $[Y\times \A^{n^2}/G]$ and $[Y/G]$ by $\Pi^1$ and $\Pi^0$. The maps $s$ and $p$ induces two morphisms $s_1:\Pi^0\to \Pi^1$ and $p_1:\Pi^1\to \Pi^0$, and that $p_1\circ s_1=\mathrm{id}$. Thus $p_1$ is surjective and $\Pi^0$ is a quotient of $\Pi^1$. Therefore, $K^{+}(\Pi^0_{(\lambda)})$ and $K^{+}(\Pi^0_{(\lambda')})$ are lambda subsemiring of $K^{+}(\Pi^1_{(\lambda)})\cong K^{+}(\Pi^1_{(\lambda')})$. 

As Theorem \ref{theorem 4.2} holds for $[Y\times \A^{n^2}/G]$, there is a unique isomorphism $\Pi^1_{(\lambda)}\simeq \Pi^1_{(\lambda')}$, which induces an isomorphism of lambda semirings: $K^{+}(\Pi^1_{(\lambda)})\cong K^{+}(\Pi^1_{(\lambda')})$. Moreover, the following diagram commutes ($A$ here is the algebra of functions from the Frobenius elements of $\Pi^1$ to $\overline{\Q}$):

\begin{center}
\label{diag 4}
    \begin{tikzcd}
    K^+(\hat{\Pi}^0_{(\lambda)})\arrow[dd,"\rotatebox{90}{$\sim$}"]\arrow[hookrightarrow]{r}& K^+(\hat{\Pi}^1_{(\lambda)})\arrow[dd,"\rotatebox{90}{$\sim$}"]\arrow[hookrightarrow]{dr}\\
    &&A\\
    K^+(\hat{\Pi}^0_{(\lambda')})\arrow[hookrightarrow]{r}&K^+(\hat{\Pi}^1_{(\lambda')})\arrow[hookrightarrow]{ur}
    \end{tikzcd}
\end{center}
Thus, by a similar idea as the proof of Lemma \ref{lemma 4.9}, the unique isomorphism $\Pi^1_{(\lambda)}\simeq \Pi^1_{(\lambda')}$ induces an isomorphism $\Pi^0_{(\lambda)}\simeq \Pi^0_{(\lambda')}$. 

Denote the induced isomorphism $\Pi^0_{(\lambda)}\simeq \Pi^0_{(\lambda')}$ as $\Psi^0$. Similar as in the proof of Lemma \ref{lemma 4.9}, given a representation $\rho_{\lambda}:\Pi\to G(\overline{\Q_{\lambda}})$, $\Psi^0$ carried it to another representation $\rho_{\lambda'}:\Pi\to G(\overline{\Q_{\lambda'}})$. Composing these two representations with a faithful representation of $G$ over $\Q$, $\theta:G\to \mathrm{GL}_n$, we can regard them as two sheaves $\mathcal{E}_{\lambda}$ and $\mathcal{E}_{\lambda'}$ on $[Y/G]$. By the commutativity of Diagram \ref{diag 4}, $p^*(\mathcal{E}_{\lambda})$ and $p^*(\mathcal{E}_{\lambda'})$ are compatible. Since the following is a Cartesian diagram:
\begin{center}
    \begin{tikzcd}
    Y\arrow[d,]\arrow[r,]& Y\times \mathbb{A}^{n^2}\arrow[d,]\\
    
    [Y/G]\arrow[r,"s"]& \mathrm{[}Y\times \mathbb{A}^{n^2}/G]
    \end{tikzcd}
\end{center}
and that the zero section $Y\to Y\times \mathbb{A}^{n^2}$ is a closed immersion, $s$ is also a closed immersion. Therefore, since $p^*(\mathcal{E}_{\lambda})$ and $p^*(\mathcal{E}_{\lambda'})$ are compatible on $[Y\times \mathbb{A}^{n^2}/G]$, $\mathcal{E}_{\lambda}=s^*p^*(\mathcal{E}_{\lambda})$ and $\mathcal{E}_{\lambda'}=s^*p^*(\mathcal{E}_{\lambda'})$ are compatible on $[Y/G]$, which proves that $\Psi^0$ makes the left half of Diagram \ref{diag 2} commute.

The fact that $\Psi^0$ makes the right half of Diagrams \ref{diag 2} can be proved in a similar way as that in Lemma \ref{lemma 4.9}. Thus, Theorem \ref{theorem 4.2} holds for quotient stack of the form $[Y/G]$, where $Y$ is affine and smooth and $G$ is a finite group scheme.

By \cite[Corollaire 6.1.1]{LMB}, every DM stack $\mathcal{X}$ of finite presentation over $\mathbb{F}_q$ has an open dense substack $\mathcal{U}$, such that $\mathcal{U}$ is isomorphic to a global quotient stack $[\Spec{A}/G]$ with $G$ being a finite group scheme. Therefore, by Lemma \ref{lemma 4.9}, Theorem \ref{theorem 4.2} holds for all smooth connected DM stacks $\mathcal{X}$ of finite presentation over $\mathbb{F}_q$ as well.
\end{proof}
\subsection{Reduction from the algebraic stacks case to DM stacks}
\begin{lemma}(Homotopy exact sequence for algebraic stacks)
\label{lemma 4.11}
    Let $\mathcal{X}$ be a geometrically integral algebraic stack of finite presentation over an arbitrary field $k$, $\overline{x}$ a geometric point of $\mathcal{X}$, and denote by $k^s$ the separable closure of $k$. Then the short sequence
    \begin{align*}
        0\to \pi_1(\mathcal{X}_{k^s},\overline{x})\to \pi_1(\mathcal{X},\overline{x})\to \mathrm{Gal}(k^{s}/k)\to 0
    \end{align*}
    is exact.
\end{lemma}
The proof relies on the following lemmas, which are developed in stages. 

\begin{lemma}
\label{lemma 4.12}
    Let $X_1$ be a scheme of finite type over a field $k$. Suppose that there is a morphism of finite type $f:X_{2,k^s}\to X_{1,k^s}$. Then there is a finite separable extension $L/k$ such that $f$ descends to $f_L:X_{2,L}\to X_{1,L}$, where $X_{1,L}=X_{1}\bigotimes_k L$, $X_{2,L}\bigotimes_L k^s=X_{2,k^s}$, $f_L\bigotimes_L k^s=f$ and $f_L$ is of finite presentation. If $f$ is moreover smooth (resp. \'etale, resp. finite \'etale, resp. a monomorphism), $f_L$ is also smooth (resp. \'etale, resp. finite \'etale, resp. a monomorphism).  Here, $X_{2,L}$ means this scheme is defined over $L$, similarly for other notations of this kind.
\end{lemma}
\begin{proof}
Since finite type morphism over a locally Noetherian base is of finite presentation and that $X_{1,k^s}$ is locally Noetherian, $f$ is of finite presentation. As $X_1$ and $X_{2,k^s}$ are quasi-compact, there exist affine open coverings $X_1=\cup_{j\in J}V_j$ and affine open coverings $\cup_{i\in I_j}U_{i,j,k^s}\subseteq f^{-1}(V_{j,k^s})$ such that the ring maps $\mathcal{O}_{X_{1,k^s}}(V_{j,k^s})\to \mathcal{O}_{X_{2,k^s}}(U_{i,j,k^s})$ is of finite presentation, for all $j\in J,i\in I_j$, with $J,I_j$ all being finite sets.

Because all these rings and gluing maps use only finitely many coefficients in $k^s$, they descend to a finite extension $L/k$.

If $f$ is \'etale, there exist a finite set of standard opens $D(h_s), s\in S$ covering $U_{i,j,k^s}$ such that the restriction of $f$ to $D(h_s)$ is standard \'etale, i.e., if $V_j=\Spec{A}$, then $U_{i,k^s}$ is covered by a finite collection of standard opens $D(h_s)$ of $U_{i,k^s}$, each having the form $\Spec{(A\bigotimes_k k^s)[x]/(g)}$. So if we take $L$ such that $A\bigotimes_k L$ contains in addition the coefficients of all $g_i$ and all $h_s$, then the restriction of $f_L$ to each $U_{i,k^s}$ is still \'etale. Since being \'etale is Zariski local on the source and target, $f_L$ is also \'etale.

If $f$ is smooth, by \cite[lemma 00TA]{stacks-project}, there exist a finite set of standard opens $D(h_s), s\in S$ covering $U_{i,k^s}$ such that the restriction of $f$ to $D(h_s)$ is standard smooth, i.e., if $V_j=\Spec{A}$, then $U_{i,k^s}$ is covered by a finite collection of standard opens $D(h_s)$ of $U_{i,k^s}$, each having the form $\Spec{(A\bigotimes_k k^s)[x_1,...,x_n]/(g_1,...,g_c)}$. So if we take $L$ such that $A\bigotimes_k L$ contain in addition the coefficients of $g_i$'s and $h_s$'s, then the restriction of $f_L$ to each $U_{i,k^s}$ is still smooth. Since being smooth is Zariski local on the source and target, $f_L$ is also smooth.

If $f$ is finite \'etale, by previous argument, $f_L$ is already \'etale, and for any finite separable field extension $L'/L$, $f_{L'}$ is still \'etale, so I only need to find a finite subextension $L'/L$ of $k^s/L$ such that $f_{L'}$ is finite. Since $f$ is of finite presentation, being finite is the same as being integral. Since there are only finitely many $V_j$ and finitely many standard opens $D(h_s)$, there is a finite extension $L'/L$ such that $f|_{D(h_s)}$ are all integral. As a result, $f_{L'}$ is finite.

If $f$ is a monomorphism, then $f_L$ is also a monomorphism, by using that being a monomorphism is equivalent to the diagonal being an isomorphism. (an isomorphism will always descend to an isomorphism)
\end{proof}
\begin{lemma}
        \label{lemma 4.13}
        Let $\mathcal{X}_{1}$ be an algebraic space of finite presentation over $k$, and let $f:\mathcal{X}_{2,k^s}\to \mathcal{X}_{1,k^s}$ be a finite \'etale morphism. Then there is a finite extension $L/k$ such that $f$ descends to $f_L:\mathcal{X}_{2,L}\to \mathcal{X}_{1,L}$, where $\mathcal{X}_{1,L}=\mathcal{X}_{1}\bigotimes_k L$, $\mathcal{X}_{2,L}\bigotimes_L k^s=\mathcal{X}_{2,k^s}$, $f_L\bigotimes_L k^s=f$ and $f_L$ is finite \'etale. If $f$ is an isomorphism, then $f_L$ is also an isomorphism.
\end{lemma}
\begin{proof}
We first choose a representation for $\mathcal{X}_1$: pick an \'etale surjection \(p:U\to\mathcal X_1\) with \(U\) a scheme of finite presentation over \(k\).
Form the relation
\[
R := U\times_{\mathcal X_1} U \rightrightarrows U,
\]
with the two projections denoted $\mathrm{pr}_1,\mathrm{pr}_2:R\to U$.

\bigskip

Then we will produce a cover and a relation scheme whose quotient will be our desired descent for $\mathcal{X}_{2,k^s}$.
Base-change everything to $k^s$ and form
\[
V := \mathcal{X}_{2,k^s}\times_{\mathcal{X}_{1,k^s}} U_{k^s}.
\]
We will use Lemma \ref{lemma 4.12} to descend $V$. 
Because \(f\) is representable and finite \'etale, the projection $V\to U_{k^s}$ is finite \'etale.
By covering $U$ by finitely many affines $U_i=\mathrm{Spec}(A_i)$ and using the fact that they are finitely presented, 
there exists a finite extension $L/k$ such that $V$ (as a finite \'etale $U_{k^s}$-scheme) is obtained
by base change from a finite \'etale $U_L$-scheme $V_L\to U_L$ (where $U_L:=U\otimes_k L$).
Concretely, for each affine $U_i$ the finite \'etale $A_i\otimes_k k^s$-algebra corresponding to $V$ 
is defined over some finite subextension of $k^s$, and a single finite extension $L$ contains all such subextensions.

For the relation scheme, we consider the base-change of the relation to $L$:
\[
R_L := R\otimes_k L \;=\; U_L\times_{\mathcal X_{1,L}} U_L \rightrightarrows U_L,
\]
with projections again $\mathrm{pr}_1,\mathrm{pr}_2:R_L\to U_L$.
Pulling $V_L\to U_L$ back along these projections gives two $R_L$-schemes
\[
\mathrm{pr}_1^*V_L \;=\; V_L\times_{U_L,\mathrm{pr}_1} R_L,
\qquad
\mathrm{pr}_2^*V_L \;=\; V_L\times_{U_L,\mathrm{pr}_2} R_L.
\]
Over $k^s$ the two pullbacks of $V$ along $R_{k^s}\rightrightarrows U_{k^s}$ are canonically isomorphic,
because $V$ was obtained by pulling $\mathcal X_{2,k^s}$ back along $U_{k^s}\to\mathcal X_{1,k^s}$.
By the same finite-presentation argument as above, that canonical isomorphism and its cocycle condition
are defined over a finite extension; after enlarging $L$ if necessary we obtain an isomorphism of $R_L$-schemes
\[
\varphi_L:\mathrm{pr}_1^*V_L \xrightarrow{\sim} \mathrm{pr}_2^*V_L
\]
which satisfies the cocycle condition on $U_L\times_{\mathcal X_{1,L}} U_L \times_{\mathcal X_{1,L}} U_L$.
Thus we are able to form the relation scheme \(R_{V,L}\) on \(V_L\).
Using $\varphi_L$ we define a relation on $V_L$ by setting
\[
R_{V,L} \;:=\; \mathrm{pr}_1^*V_L \;\cong_{\varphi_L}\; \mathrm{pr}_2^*V_L
\qquad\text{with the two structure maps } s,t: R_{V,L}\rightrightarrows V_L
\]
given by the projections to the two factors of $\mathrm{pr}_i^*V_L\cong V_L\times_{U_L}R_L$.
Explicitly, a point of $R_{V,L}$ is a triple $(v_1,r,v_2)$ with $v_i\in V_L$, $r\in R_L$ and
$\mathrm{pr}_i(r)=\pi(v_i)$ (where $\pi:V_L\to U_L$ is the structure map), and the identification
$\varphi_L$ equates the two descriptions coming from $\mathrm{pr}_1^*V_L$ and $\mathrm{pr}_2^*V_L$.
The cocycle condition for $\varphi_L$ ensures that $R_{V,L}$ is an equivalence relation.

\bigskip

Now we are able to explicitly construct the descent $\mathcal{X}_{2,L}$. Define the presheaf (on the big fppf site over $\mathcal{X}_{1,L}$)
\[
\mathcal{F} \;=\; V_L / R_{V,L},
\]
Let $\mathcal{X}_{2,L}$ denote the fppf sheafification of $\mathcal F$.

Because $R_L\overset{\mathrm{pr}_1}{\underset{\mathrm{pr}_2}{\rightrightarrows}} U_L$ is an \'etale equivalence relation, the relation $R_{V,L}\rightrightarrows V_L$ is also an \'etale equivalence relation, as it is the pullback of $R_L\overset{\mathrm{pr}_1}{\underset{\mathrm{pr}_2}{\rightrightarrows}} U_L$ along $V_L\to U_L$. It is clear that the fppf sheafification is representable by an algebraic space; thus $\mathcal{X}_{2,L}$ is an algebraic space and the canonical map
\[
q:V_L \to \mathcal{X}_{2,L}
\]
is an \'etale surjection presenting $\mathcal{X}_{2,L}$ as the quotient of $V_L$ by $R_{V,L}$.

Finally we construct the map \(\mathcal X_{2,L}\to\mathcal X_{1,L}\) and check its properties.
By our construction above, the map $V_L \xrightarrow{} U_L$ descends to a map of algebraic spaces
\[
f_L:\mathcal{X}_{2,L} \longrightarrow \mathcal X_{1,L}.
\]
Locally on $\mathcal X_{1,L}$ (pull back along the \'etale cover $U_L\to\mathcal X_{1,L}$) this map is identified with $V_L\to U_L$,
which is finite \'etale; therefore $f_L$ is representable and finite \'etale. Finally, by construction all objects and identifications were chosen so that
base change along $L\hookrightarrow k^s$ recovers the original $V\to U_{k^s}$, the original relation on it, and hence the original
algebraic space $\mathcal{X}_{2,k^s}$ and morphism $f$. Thus $f_L\otimes_L k^s \cong f$ as required.    
\end{proof}

\begin{lemma}
\label{lemma 4.14}
    Lemma \ref{lemma 4.13} is also true when $\mathcal{X}_1$ is an algebraic stack of finite presentation over $k$, and $f$ is representable by algebraic spaces. (Here, printed letter denote algebraic spaces while curly letter denote stacks)
\end{lemma}
\begin{proof}
We proceed in a manner similar to the proof of the previous lemma. We first choose a smooth presentation for \(\mathcal X_1\):
pick a smooth surjection \(p:U\to\mathcal{X}_1\) with \(U\) a scheme of finite presentation over $k$. 
Let 
\[
R:=U\times_{\mathcal{X}_1}U\rightrightarrows U
\]
be the associated relation algebraic space with projection maps $\mathrm{pr}_1,\mathrm{pr}_2:R\to U$. 

\bigskip

Then we will produce a cover and a relation algebraic space whose quotient will be our desired descent for $\mathcal{X}_{2,k^s}$.
Form the fiber product
\[
V := \mathcal{X}_{2,k^s}\times_{\mathcal{X}_{1,k^s}} U_{k^s}.
\]
Because $f$ is representable and finite \'etale, the map $V\to U_{k^s}$ is a finite \'etale morphism of algebraic spaces. 
By the previous Lemma \ref{lemma 4.13}, we can choose a finite extension $L/k$ so that $V$ is obtained by base change from a finite \'etale $U_L$-algebraic space $V_L \to U_L$, where $U_L := U\otimes_k L$.

For the relation algebraic space, let $R_L := R\otimes_k L = U_L\times_{\mathcal{X}_{1,L}}U_L \rightrightarrows U_L$ with projections $\mathrm{pr}_1,\mathrm{pr}_2:R_L\to U_L$.
Pull back \(V_L\to U_L\) along these projections to obtain two $R_L$-spaces
\[
\mathrm{pr}_1^*V_L,\qquad \mathrm{pr}_2^*V_L.
\]
Over $k^s$ these two pullbacks are canonically identified because $V$ came from pulling back $\mathcal{X}_{2,k^s}$; by finite-presentation-ness the isomorphism and its cocycle condition descend to $L$ after enlarging $L$ if necessary. Thus there exists an isomorphism of $R_L$-spaces
\[
\varphi_L:\mathrm{pr}_1^*V_L\overset{\sim}{\longrightarrow}\mathrm{pr}_2^*V_L
\]
satisfying the cocycle condition on triple overlaps.

Using $\varphi_L$ we form the groupoid
\[
R_{V,L}\;:=\; \mathrm{pr}_1^*V_L \;\cong_{\varphi_L}\; \mathrm{pr}_2^*V_L
\;\rightrightarrows\; V_L,
\]
with source and target maps $s,t:R_{V,L}\rightrightarrows V_L$ given by the two projections to the two factors of $\mathrm{pr}_i^*V_L\cong V_L\times_{U_L}R_L$.  The cocycle condition guarantees the groupoid axioms hold for $R_{V,L}$.

\bigskip

Now we are able to explicitly construct the descent $\mathcal{X}_{2,L}$. We stackify the prestack $[V_L/R_{V,L}]^{pre}$ over the big fppf site to obtain a stack
\[
\mathcal{X}_{2,L} \;:=\; [\,V_L / R_{V,L}\,].
\]

Because $R_{V,L}\rightrightarrows V_L$ is a smooth equivalence relation, the stack quotient is an algebraic stack.

\bigskip

Finally we construct the map \(\mathcal{X}_{2,L}\to\mathcal X_{1,L}\) and check its properties. By our construction, the map $V_L\to U_L$ induces a morphism of stacks
\[
f_L:\mathcal{X}_{2,L}=[V_L/R_{V,L}]\longrightarrow \mathcal{X}_{1,L}.
\]
It is clear that all of the objects ($V_L$, $R_{V,L}$ and the isomorphism $\varphi_L$) were chosen so that their base change to $k^s$ recovers the original objects $V$, $R$ and the canonical identification. Therefore
\[
\mathcal{X}_{2,L}\times_L k^s \;=\; [\,V_L\times_L k^s \;/\; R_{V,L}\times_L k^s\,] \;\cong\; [\,V\;/\; R\,] \;\cong\; \mathcal{X}_{2,k^s},
\]
and $f_L\otimes_L k^s=f$, thus the morphism $f_L$ is representable by algebraic spaces, and since $V_L\to U_L$ is finite \'etale, $f_L$ is finite \'etale as well.

\end{proof}

\begin{proof}[Proof of Lemma \ref{lemma 4.11}]

By Galois formalism, the injectivity follows from Lemma \ref{lemma 4.14}, and surjectivity follows from the fact that $\mathcal{X}_{k}$ is geometrically integral. 

We check middle exactness by the usual Galois formalism, \cite[lemma 0BS9, 0BTS]{stacks-project}, it suffices to verify:
    \begin{enumerate}
        \item For every finite separable extension $L/k$, $\mathrm{Spec}(L)\times_{\mathrm{Spec}(k)}\mathcal{X}_{k^s}$ is a trivial cover of $\mathcal{X}_{k^s}$ (this is immediate since $L\otimes_k k^s$ is a product of copies of $k^s$);
        \item For a finite Galois cover $\mathcal{Y}\to \mathcal{X}$, if $\mathcal{Y}_{k^s}\to \mathcal{X}_{k^s}$ is a trivial cover, then $\mathcal{Y}$ is dominated by $\mathcal{X}_A$ for some finite \'etale algebra $A$ over $k$;
        \item For every connected finite \'etale cover $\mathcal{Y}\to \mathcal{X}$ whose pullback to $\mathcal{X}_{k^s}$ admits a section, the cover $\mathcal{Y}_{k^s}\to \mathcal{X}_{k^s}$ is a disjoint union of trivial covers.
    \end{enumerate}

    We treat $(2)$ and $(3)$. For $(2)$, let $\mathcal Y\to\mathcal X$ be finite (Galois) with $\mathcal Y_{k^s}\cong\coprod_{i\in I}\mathcal X_{k^s}$. Choose a smooth surjective presentation $p:X\to\mathcal X$ with $X$ irreducible (and reduced). By replacing $k$ by a finite separable extension if necessary, we may and do assume $X$ is geometrically integral over $k$ (this is standard: an integral finite type $k$-scheme becomes geometrically integral after some finite separable base change).
    
    We have the following diagram: 
    \begin{center}
    \begin{tikzcd}    \mathcal{Y}_{k^s}\times_{\mathcal{X}_{k^s}}X_{k^s}\arrow[dr,]\arrow[rrr,]\arrow[ddd,]& & &\mathcal{Y}\times_{\mathcal{X}}X \arrow[dl,]\arrow[ddd,]\\
    &\mathcal{Y}_{k^s}\arrow[r,]\arrow[d,]& \mathcal{Y}\arrow[d,]&\\    
    &\mathcal{X}_{k^s}\arrow[r,]\arrow[u,,shift left=1ex,dashed]& \mathcal{X}&\\ 
    X_{k^s}\arrow[rrr,]\arrow[uuu,,shift left=1ex,dashed]\arrow[ur,]& & &X\arrow[ul,]
    \end{tikzcd}
    \end{center}

Since $\mathcal{Y}_{k^s}=\amalg_{i\in I}\mathcal{X}_{k^s}$, $\mathcal{Y}_{k^s}\times_{\mathcal{X}_{k^s}}X_{k^s}=\amalg_{i\in I}X_{k^s}$, for some finite set $I$. So every connected component of $\mathcal{Y}\times_{\mathcal{X}}X$ is trivialized after pulling back to $X_{k^s}$, which means that there exists some finite separable extension $k_1/k$ such that $\mathcal{Y}_{k_1}\times_{\mathcal{X}_{k_1}}X_{k_1}=\amalg_{j\in J}X_{k_1}$, for some finite set $J$. 

The relation algebraic space for $\mathcal{Y}_{k_1}$ is $R_{\mathcal{Y}_{k_1}}=R_{\mathcal{X}_{k_1}}\times_{X_{k_1}}\amalg_{j\in J}X_{k_1}=\amalg_{j\in J}(R_{\mathcal{X}})_{k_1}$. Therefore, $\mathcal{Y}_{k_1}=\amalg_{j\in J}[X_{k_1}/(R_{\mathcal{X}})_{k_1}]=\amalg_{j\in J}\mathcal{X}_{k_1}$, as stackification commutes with fiber products. Thus $\mathcal{Y}$ is dominated by $\mathcal{X}_A$ for some finite \'etale algebra over $k$.

For (3), we notice that if we pull back a connected finite \'etale cover $\mathcal{Y}\to \mathcal{X}$ to $\mathcal{X}_{k^s}$, it will split into a disjoint union of connected finite \'etale covers of $\mathcal{X}_{k^s}$, and that for some finite separable extension $k_2/k$, the Galois group $\mathrm{Gal}(k_2/k)$ acts transitively on the connected components of the disjoint union. So if component is a trivial cover, so are others. However, $\mathcal{Y}_{k^s}\to \mathcal{X}_{k^s}$ admits a section, which means that at least one component is a trivial cover, so the cover $\mathcal{Y}_{k^s}\to \mathcal{X}_{k^s}$ is a disjoint union of trivial covers. Therefore, the exactness in the middle is proved and the proof is complete.

\end{proof}

In \cite{noohi2002fundamental}, Noohi introduced monotonous gerbe, which is defined as follows:
\begin{definition}
    We say that a finite type flat group space $G \to S$ is monotonous, if the number of geometric connected components of the fiber of point $s \in S$ is a locally constant function on $S$, and if the connected component functor $G^{\circ}$ is representable (necessarily by an open subspace of $G$). We say that an algebraic stack $\mathcal{X}$ is a monotonous gerbe, if it is connected and its stabilizer group is monotonous.
\end{definition}

\begin{proposition}
\label{proposition 4.16}
    For a monotonous gerbe $\mathcal{X}$, the map constructed in \cite[Propsition 5.1.11]{Behrend} $\Phi: \mathcal{X}\to \overline{\mathcal{X}}$ induces an isomorphism of fundamental groups: $\pi_1(\mathcal{X},x)\cong \pi_1(\overline{\mathcal{X}},y)$, where $x$ is a geometric point over some closed point in $\mathcal{X}$, and $y$ is the composition of $x$ and $\Phi$.

    Moreover, if Theorem \ref{theorem 4.2} holds for DM stacks, then it holds for monotonous stacks.
    
\end{proposition}
\begin{proof}

By \cite[Corollaire 10.8]{LMB}, a monotonous gerbe is an fppf gerbe, hence the diagonal map of $\mathcal{X}\to X_{\mathrm{mod}}$ is fppf.

Moreover, by definition, the stabilizer group space is monotonous. Namely, given an fppf presentation $X\to \mathcal{X}$, the group space $S_{X}\to X$ in the following diagram is monotonous.
\begin{center}
    \begin{tikzcd}
    S_X\arrow[d,]\arrow[r,]& \mathcal{I}_{\mathcal{X}}\arrow[d,]\arrow[r,]&\mathcal{X}\arrow[d,"\triangle"]\\
    
    X\arrow[r,]& \mathcal{X}\arrow[r,"\triangle"]&\mathcal{X}\times\mathcal{X}
    \end{tikzcd}
\end{center}
By \cite[Corollary 5.1.9]{Behrend}, this means that $S_X\to X$ has components. Since having component is local on the base with respect to the fppf-topology, $\mathcal{I}_{\mathcal{X}}\to \mathcal{X}$ also has components. Since $X_{\mathrm{mod}}$ is an algebraic space, $\mathcal{I}_{\mathcal{X}}$ can be naturally identified with $\mathcal{I}_{\mathcal{X}/X_{\mathrm{mod}}}$. Since the relative inertia stack is the pullback of the diagonal map of $\mathcal{X}\to X_{\mathrm{mod}}$ along the same diagonal map, the diagonal map has components. Thus, applying \cite[Proposition 5.1.11]{Behrend} to $\mathcal{X}\to X_{\mathrm{mod}}$, we have a factorization:
\begin{align}
    \mathcal{X}\xrightarrow{g_0} \overline{\mathcal{X}}\xrightarrow{h_0} X_{\mathrm{mod}}\text{ ,}
\end{align}
where $g$ and $h$ are gerbe-like, $h$ is \'etale and $g$ is diagonally connected.

In the following diagram, $g$ and $h$ are induced by $g_0$ and $h_0$ respectively, and $\pi_1^h(\mathcal{X},x)$ means the hidden fundamental group, whose definition is in \cite[Definition 3.1]{noohi2002fundamental}:
\begin{center}
\begin{tikzcd}
&\pi_1^h(\mathcal{X},x)\arrow[r,"\phi"]\arrow[d,"f"]&\pi_1(\mathcal{X},x)\arrow[d,"g"]\arrow[r,"\psi"]&\pi_1(X_{\mathrm{mod}},z)\arrow[d,]\arrow[r,]&0\\
0\arrow[r,]&\pi_1^h(\overline{\mathcal{X}},y)\arrow[r,"i"]&\pi_1(\overline{\mathcal{X}},y)\arrow[r,"h"]&\pi_1(X_{\mathrm{mod}},z)\arrow[r,]&0,
\end{tikzcd}
\end{center}
where the last vertical map is the identity map. Moreover, by \cite[corollary 5.4]{noohi2002fundamental}, $i$ is injective.

Now assume that $\mathcal{X}$ is defined over an algebraically closed field. To study $f$, we first note that $\pi_1^h(\mathcal{X},x)$ and $\pi_1^h(\overline{\mathcal{X}},y)$ can be identified as $k$-points of the stabilizer group scheme of $x$ and the $k$-points of the stabilizer group scheme of $y$, respectively (by \cite[Proposition 3.4]{noohi2002fundamental}). Moreover, for classifying stacks $B(G/S)$, where $G$ is a flat group space of finite presentation over $S$, the factorization Behrend constructed for it is $BG \to B(G/G^{\circ}) \to S$. Which shows that when $\mathcal{X}=B(G/S)$, the map $f$ is indeed induced by $G_x\to (G_x/G^{\circ}_x)$, so $f$ is surjective. To generalize that, it suffices to study the case where $\mathcal{X}$ is a residual gerbe. Since residual gerbes are \'etale locally classifying stacks, and that we assumed that $\mathcal{X}$ is defined over an algebraically closed field so far, $f$ will always be induced by the map quotienting the stabilizer group scheme by its identity component, which implies that $f$ is surjective.

Since $\phi$ also factor through $G/G^{\circ}$, (because the factorization of the residual gerbe), we can replace $\pi_1^h(\mathcal{X},x)$ by $G/G^{\circ}$, and call the vertical map as $f'$ and the map from $G/G^{\circ}$ to $\pi_1(\overline{\mathcal{X}},y)$ as $\phi'$, as follows:

\begin{center}
\begin{tikzcd}
&G/G^{\circ}\arrow[r,"\phi'"]\arrow[d,"f'"]&\pi_1(\mathcal{X},x)\arrow[d,"g"]\arrow[r,"\psi"]&\pi_1(X_{\mathrm{mod}},z)\arrow[d,]\arrow[r,]&0\\
0\arrow[r,]&\pi_1^h(\overline{\mathcal{X}},y)\arrow[r,"i"]&\pi_1(\overline{\mathcal{X}},y)\arrow[r,"h"]&\pi_1(X_{\mathrm{mod}},z)\arrow[r,]&0.
\end{tikzcd}
\end{center}

Since both $f'$ and $i$ are injective, by a simple diagram chase, $g$ is also injective. Thus, $g$ is an isomorphism of groups.

In the more general case where the base field is not algebraically closed, since both $g_0$ and $h_0$ are defined over the base field, by the homotopy exact sequence for algebraic stacks, $g$ is still an isomorphism.

For the last claim, to simplify notation, denote by $\Pi^1$ the \'etale fundamental group of $\mathcal{X}$, and denote by $\Pi^0$ the \'etale fundamental group of $\overline{\mathcal{X}}$. $g:\Pi^1\xrightarrow{\sim}\Pi^0$ induces an isomorphism of $\widehat{\Pi}^1_{(\lambda)}$ and $\widehat{\Pi}^0_{(\lambda)}$. 

Since the canonical morphism $\Pi^1\to \widehat{\Pi}^1_{\lambda}$ from the universal property of inverse limits is a section to the canonical map $\widehat{\Pi}^1_{\lambda}\to \pi_0(\widehat{\Pi}^1_{\lambda})=\Pi^1$ and that functors preserve commutativity of diagrams, $g$ makes the following diagram commutative:
\begin{center}
    \begin{tikzcd}
        \widehat{\Pi}^1_{(\lambda)}\arrow[r,""] \arrow[d,""] &\Pi^1\arrow[d,"\mathbf{Pro-ss}(g)"]\\
        \widehat{\Pi}^0_{(\lambda)}\arrow[r,""] &\Pi^0
    \end{tikzcd}
\end{center}

On the other hand, given $\lambda$ and $\lambda'$, and a representation $\rho_{\lambda}:\Pi^1\to G(\overline{\Q_{\lambda}})$, $g^{-1}$ carried it to a representation $\overline{\rho}_{\lambda}:\Pi^0\to G(\overline{\Q_{\lambda}})$. By Theorem \ref{theorem 4.10}, $\overline{\rho}_{\lambda}$ has a companion $\overline{\rho}_{\lambda'}$. If we pull back $\overline{\rho}_{\lambda'}$ by $g$, we got a $\lambda'$-representation of $\Pi^1$, which is a companion of $\rho_{\lambda}$
since $\Phi: \mathcal{X}\to \overline{\mathcal{X}}$ is surjective. So the left half of Diagram \ref{diag 2} is commutative when $\mathcal{X}$ is a monotonous stack.

Therefore, if Theorem \ref{theorem 4.2} holds for DM stacks, it holds for monotonous stacks as well.
\end{proof}

\begin{proof}[Proof of Theorem \ref{theorem 4.2}:]
By \cite[theorem 11.3]{noohi2002fundamental}, for every Noetherian algebraic stack $\mathcal{X}$, there exists a disjoint union of open substack $\amalg_i\mathcal{U}_i$ such that $\amalg_i\mathcal{U}_i$ is open and dense in $\mathcal{X}$ and that each $\mathcal{U}_i$ with its reduced structure is a monotonous gerbe. Thus, by Proposition \ref{proposition 4.16} and Lemma \ref{lemma 4.9}, we are reduced to the DM stack case. Combined with Theorem \ref{theorem 4.10}, this shows that Theorem \ref{theorem 4.2} holds for all smooth algebraic stacks of finite presentation over $\mathbb{F}_q$.

To prove the uniqueness, we can apply Drinfeld's idea. First notice that by \cite[Proposition 3.1.4]{drinfeld2018prosemisimple} and the fact that $\widehat{\Pi}_\ell$ is the maximal pro-semisimple quotient of the
$\mathbb{Q}_\ell$-pro-algebraic completion of $\Pi$, $\widehat{\Pi}_{\ell}^{\circ}$ is simply connected. Since $\widehat{\Pi}_\lambda := \widehat{\Pi}_\ell \otimes_{\mathbb{Q}_\ell} \overline{\mathbb{Q}_\lambda}$ by Sublemma \ref{Sublemma 4.4},  $\widehat{\Pi}_\lambda^{\circ}$ is also simply connected. Thus $\widehat{\Pi}_{(\lambda)}^{\circ}$ is also simply connected and is a product of almost-simple groups. \cite[Proposition 2.4.4]{drinfeld2018prosemisimple} implies that $\widehat{\Pi}_{(\lambda)}$ has no nontrivial automorphisms inducing the identity on
$[\widehat{\Pi}_{(\lambda)}]$, and the uniqueness follows.
\end{proof}

\section{Main theorem for reductive $G$}
Before stating the main theorem in the reductive case, we define lisse sheaves that are plain of characteristic $p$ and those that are $G$-irreducible if they have value group $G$ (viewed as a representation into $G$). We continue to denote the \'etale fundamental group of $\mathcal{X}$ as $\Pi$.

\begin{definition}
\label{definition 5.1}
Let $\alpha\in \overline{\Q}$, $\alpha$ is said to be plain of characteristic $p$ if $\alpha$ is a $\lambda$-adic unit for all places in $\Q(\alpha)$ such that $\lambda\nmid p$. 

Let $\mathcal{X}$ be an algebraic stack over a finite field $\kappa$ of characteristic $p$. Let $\lambda$ be a place in $\overline{\Q}$ not dividing $p$. A $\overline{\Q_{\lambda}}$-lisse sheaf $\mathcal{E}$ on $\mathcal{X}$ is called plain of characteristic $p$ if for all finite extension $\kappa'/\kappa$, the Frobenius eigenvalues of the points $x\in \mathcal{X}(\kappa')$ are plain of characteristic $p$. We can also equivalently think of $\mathcal{E}$ as a continuous representation $\Pi\to \mathrm{GL}_n(\overline{\Q_\ell})$, and this definition can be applied to such representations. 

A lisse sheaf valued in $G$, i.e.\ corresponding to a representation
\[
\rho : \Pi \longrightarrow G(\overline{\mathbb{Q}}_\lambda),
\]
is said to be $G$-irreducible if the representation does not factor through any proper parabolic
subgroup of $G$.
\end{definition}
This definition can also be applied to a representation with reductive target group via the following proposition:
\begin{proposition}
\label{proposition 5.2}
    Let $G$ be a reductive group over $\Q$ and let $\rho:\Pi\to G(\overline{\Q_\ell})$ be a continuous representation. The following are equivalent. Here, all representations of \(G\) are understood to be finite-dimensional algebraic representations.
    \begin{enumerate}
        \item[(1)] For any irreducible representation $\theta$ of $G$, $\theta\circ \rho$ is plain of characteristic $p$;
        \item[(2)] For any faithful representation $\theta$ of $G$, $\theta\circ \rho$ is plain of characteristic $p$;
        \item[(3)] There exists a faithful representation $\theta$ of $G$, $\theta\circ \rho$ is plain of characteristic $p$;
    \end{enumerate}
\end{proposition}
\begin{proof}
    To prove that $(1)\Longrightarrow (2)$, we can use the fact that every finite dimensional representation of a reductive group is a sum of irreducible representations. For the converse, suppose that $\theta_1$ is an irreducible representation of $G$, and that $\iota$ is a faithful representation of $G$, then $\theta_1\bigoplus\iota$ is also a faithful representation of $G$. So $(2)$ implies that $\theta_1\circ \rho$ is plain of characteristic $p$. Since the choice of $\theta_1$ is arbitrary, $(2)$ implies $(1)$.

    It is clear that $(2)$ implies $(3)$. For the converse direction, we use the theorem \cite[theorem 5.44]{Milne_AG}, which states that every finite-dimensional representation of an algebraic group can be constructed from a faithful representation by forming tensor products, direct sums, duals, and subquotients. But it is clear that all the above operations preserve the property of being plain of characteristic $p$. Thus, the equivalence of $(2)$ and $(3)$ are established.
\end{proof}
\begin{definition}
    With the above notations, a representation $\rho:\Pi\to G(\overline{\Q_\ell})$ is plain of characteristic $p$ if either of the above condition $(1)$ or the condition $(2)$ holds.
\end{definition}
\begin{theorem}
\label{theorem 5.4}
    Let $\mathcal{X}$ be a smooth algebraic stack of finite presentation over $\mathbb{F}_q$. Let $\lambda$ and $\lambda'$ be non-Archimedean places of $\overline{\Q}$ not dividing $p$. Let $G$ be a (not necessarily connected) reductive group over $\Q$ and let $\rho_{\lambda}:\Pi\to G(\overline{\Q_{\lambda}})$ be an irreducible representation, and let $\rho^{ab}_{\lambda}$ be the map $\Pi\xrightarrow{\rho_{\lambda}} G(\overline{\Q_{\lambda}})\to (G/[G^{\circ},G^{\circ}])(\overline{\Q_{\lambda}})$ where the latter arrow is the canonical projection. If for some closed point $x\in \mathcal{X}$ the eigenvalues of $\rho_{\lambda}^{ab}(\mathrm{Frob}_x)$ are plain of characteristic $p$, then there exists an irreducible representation $\rho_{\lambda'}:\Pi\to G(\overline{\Q_{\lambda'}})$ which is a companion of $\rho_{\lambda}$ in the sense of Definition \ref{definition 3.7} and such that $\rho_{\lambda'}^{ab}$ is compatible with $\rho_{\lambda}^{ab}$.
\end{theorem}

The curve case of this theorem is \cite[theorem 4.6]{Independence_of_monodromy_groups}. To prove that in the stack case, we first prove an extension of Deligne's theorem to stacks:
\begin{lemma}
\label{lemma 5.5}
    Let $\mathcal{X}$ be a normal geometrically integral algebraic stack of finite presentation over $\mathbb{F}_q$ and let $x$ be a geometric point. Let $\chi : \pi_1(\mathcal{X}, x) \to \overline{\Q_\ell}^*$ be a continuous character of the fundamental group. Then the image $\chi(\pi_1(\mathcal{X}_{\overline{\mathbb{F}_q}}, x))$ of the geometric fundamental group is a finite group.
\end{lemma}
\begin{proof}
    By Lemma \ref{lemma 4.7}, combined with the fact that open substacks of normal geometrically integral algebraic stacks over $\mathbb{F}_q$ are also normal and geometrically integral, we can freely shrink to an open substack. By Proposition \ref{proposition 4.16}, Theorem \cite[theorem 11.3]{noohi2002fundamental} and the fact that $\mathcal{X}$ is geometrically integral, there exists an open substack $\mathcal{U}$ which is open and dense in $\mathcal{X}$ and that $\mathcal{U}$ with its reduced structure is a monotonous gerbe, and that there is a gerbe-like morphism $f:\mathcal{U}\to \mathcal{Y}$, where $Y$ is a DM stack and $f$ induces an isomorphism $\pi_1(\mathcal{U})\xrightarrow{\sim} \pi_1(\mathcal{Y})$. Since a gerbe-like morphism is a universal homeomorphism, $\mathcal{Y}$ is geometrically irreducible. Moreover, as $f$ is smooth and of finite type, $\mathcal{Y}$ is normal. Thus, we are reduced to the DM stack case. 

    By \cite[Corollaire 6.1.1]{LMB}, every DM stack $\mathcal{X}$ of finite presentation over $\mathbb{F}_q$ has an open dense substack $\mathcal{U}$, such that $\mathcal{U}$ is isomorphic to a global quotient stack $[Y/G]$ with $G$ being a finite group scheme and $Y$ is normal and geometrically irreducible. So we are again reduced to the case of a global quotient DM stack of the above form: $[Y/G]$. The presentation morphism $Y\to [Y/G]$ induces an injection of fundamental groups $\pi_1(Y)\to \pi_1([Y/G])$ such that the image of $\pi_1(Y)$ in $\pi_1([Y/G])$ is a finite index subgroup, thus the image of $\pi_1(Y_{\overline{\mathbb{F}_q}})$ in $\pi_1([Y/G]_{\overline{\mathbb{F}_q}})$ is also finite. Since $Y$ is a normal and geometrically irreducible scheme over $\mathbb{F}_q$, $\chi(\pi_1(Y_{\overline{\mathbb{F}_q}}))$ is finite. Thus, $\chi(\pi_1([Y/G]_{\overline{\mathbb{F}_q}}))$ is also finite, proving the lemma.    
\end{proof}

For any non-Archimedean place $\lambda$ of $\overline{\mathbb{Q}}$ not dividing $p$, we define a group $\hat{\Pi}_{\lambda}^{\text{red}}$ as  $\hat{\Pi}_{\lambda}^{\text{red}}:=\hat{\Pi}^{\text{red}}_{\ell}\bigotimes_{\Q_\ell}\overline{\Q_{\lambda}}$, where $\hat{\Pi}^{\text{red}}_{\ell}$ is the projective limit of pairs $(G,\rho)$, where $G$ is a reductive algebraic group over $\Q_\ell$ and $\rho:\Pi\to G(\Q_\ell)$ is a continuous homomorphism with Zariski dense image.

Let $\text{Plain}_p \subset \overline{\mathbb{Q}}^{\times}$ be the multiplicative group of algebraic numbers plain of characteristic $p$. Since all roots of unity are plain of characteristic $p$, $\mu_{\infty}(\mathbb{Q}) \subset \text{Plain}_p$.

Regarding $\text{Plain}_p$ as a constant group scheme, we can define another group scheme $\mathcal{A}$ over $\overline{\Q}$ by $\mathcal{A}:= \text{Hom}(\text{Plain}_p, \mathbb{G}_m)$. And for each non-Archimedean place $\lambda$ of $\overline{\mathbb{Q}}$ not dividing $p$, let $\mathcal{A}_{\lambda} := \mathcal{A} \otimes_{\overline{\Q}} \mathbb{Q}_{\lambda}$. Since $\mu_{\infty}(\overline{\Q})$ is a subgroup of $\text{Plain}_p$, $\mathcal{A}$ maps onto $\text{Hom}(\mu_{\infty}(\overline{\Q}), \mathbb{G}_m) = \text{Hom}(\varinjlim_n\mu_n(\overline{\Q}), \mathbb{G}_m) = \varprojlim_n(\Z/n\Z)=\hat{\Z}$. 

By Lemma \ref{lemma 4.11}, there is a canonical homomorphism
\[
\Pi \to \mathrm{Gal}(\overline{\mathbb{F}}_p/\mathbb{F}_p) = \hat{\mathbb{Z}}
\]
with finite cokernel. Composing it with the epimorphism $\hat{\Pi}_{\lambda} \twoheadrightarrow \Pi$, one obtains a homomorphism $\hat{\Pi}_{\lambda} \to \hat{\mathbb{Z}}$.

Now, define $\hat{\Pi}^{\text{plain of char }p}_{\lambda}$ to be the fiber product of $\hat{\Pi}_{\lambda}$ and $\mathcal{A}_{\lambda}$ over $\hat{\mathbb{Z}}$.

Similarly, define $\hat{\Pi}^{\text{plain of char }p} \in \mathrm{Pro\text{-}red}(\overline{\mathbb{Q}})$ to be the fiber product of $\hat{\Pi}$ and $\mathcal{A}$ over $\hat{\mathbb{Z}}$. Since the images of $\hat{\Pi}$ and $\hat{\Pi}_{\lambda}$ in $\mathbf{Pro\text{-}ss}(\overline{\Q_{\lambda}})$ are canonically isomorphic, the images of $\hat{\Pi}^{\text{plain of char }p}$ and $\hat{\Pi}^{\text{plain of char }p}_{\lambda}$ in $\mathrm{Pro\text{-}red}(\overline{\mathbb{Q}_{\lambda}})$ are canonically isomorphic. (This proves the compatibility)

\vskip 1cm

We can apply the construction in \cite[\textsection 3.5]{drinfeld2018prosemisimple} with $A=\overline{\Z_{\lambda}}^{\times}$ and $G=\hat{\Pi}^{\text{red}}_{\lambda}$, and we get a map
\begin{align*}
    \hat{\Pi}^{\text{red}}_{\lambda}\to \hat{\Pi}_{\lambda}\times_{\hat{\Z}}\operatorname{Hom}(\overline{\Z_{\lambda}}^{\times},\mathbb{G}_m)
\end{align*}

\begin{lemma}
\label{lemma 5.6}
The above map is an isomorphism.    
\end{lemma}
\begin{proof}
    To ease notation, set $A: =\overline{\Z_{\lambda}}^{\times}$, and $G: =\hat{\Pi}^{\text{red}}_{\lambda}$. Following Drinfeld's notation in his lemma 3.5.1, let $Z$ be the center of $G^{\circ}$. In the construction of $G^{\circ}$, since all the representations have Zariski dense image, any structure map between two $\rho$'s is surjective. So centers are mapped into centers, and $Z$ is the projective limit of the centers of the components of $G^{\circ}$. Then there are two exact sequences:
\[
0 \to Z^\circ \to G^\circ \to G^{ss} \to 0,
\]
\[
0 \to \text{Hom}(A/A_{\text{tors}}, \mathbb{G}_m) \to G^{ss} \times_{\text{Hom}(A_{\text{tors}}, \mathbb{G}_m)} \text{Hom}(A, \mathbb{G}_m) \to G^{ss} \to 0.
\]

    By the short five lemma, it suffices to prove that the composition  $Z^\circ \to G^\circ \to \text{Hom}(A/A_{\text{tors}}, \mathbb{G}_m)$ is an isomorphism. This is equivalent to showing that the composition
\[
A/A_{\text{tors}} \xrightarrow{f} \text{Hom}(G^\circ, \mathbb{G}_m) \xrightarrow{g} \text{Hom}(Z^\circ, \mathbb{G}_m)
\]
is an isomorphism.

Note that $f$ is induced by the following map: 
\begin{align*}
    &A \simeq \operatorname{Hom}_{\text{cont}}(\widehat{\mathbb{Z}}, \overline{\mathbb{Q}}_{\lambda}^{\times}) \xrightarrow{\text{inflation}} \operatorname{Hom}_{\text{cont}}(\Pi, \overline{\mathbb{Q}}_{\lambda}^{\times})= \operatorname{Hom}(G, \mathbb{G}_m)\xrightarrow{\text{res}}\operatorname{Hom}(G^{\circ}, \mathbb{G}_m)\\
    &\alpha \longrightarrow (1\to \alpha) 
\end{align*}

Since the image of a torsion element in $A$ is a finite order character $\chi$ of $G$, the restriction of $\chi$ to $G^{\circ}$ is trivial, as the image of a connected group is also connected. Moreover, if we take a closer look at the above composition of maps, we observe that the image $\chi$ of $\alpha$ in $\operatorname{Hom}(G^{\circ}, \mathbb{G}_m)$ is the following: for a representation $\rho: \Pi'\to H$, where $H$ is a connected reductive group and $\Pi'$ is an open subgroup of $\Pi$, $\chi$ is just the restriction of $\chi_{\alpha}$ to $\Pi'$, where $\chi_{\alpha}$ is the inflation of the character $1\to \alpha$ to $\Pi$.

To prove the injectivity of $f$, suppose that $\alpha_1,\alpha_2$ are two elements in $A$ which induces the same character $\chi$, and denote by $\chi_1,\chi_2$ the image of $\alpha_1,\alpha_2$ in $G$ respectively. It is clear that $\chi_1,\chi_2$ will factor through some reductive algebraic group $G_1,G_2$. Since $\chi_1,\chi_2$ agree on $G^{\circ}$, there is some integer $N$ such that $\chi_1^N=\chi_2^N$, which means that $\alpha_1=\alpha_2$ modulo $A_{\mathrm{tors}}$.

To prove the surjectivity of $f$, suppose that $\chi$ is an element in $\operatorname{Hom}(G^{\circ}, \mathbb{G}_m)$. By the analysis of the explicit structure of $f$, $\chi$ can be viewed as a character of $\Pi'$, an open subgroup of $\Pi$. By Lemma \ref{lemma 5.5}, there is an integer $n$ such that $\chi^n$ factor through the absolute Galois group. By the above analysis again, there is an $\alpha$ such that $\chi^n$ is the image of $\alpha$. This means that the cokernel of $f$ is torsion. Since $f$ is injective and that $A/A_{\mathrm{tors}}$ is a $\Q$-vector space, $f$ is indeed surjective. Thus, $f$ is an isomorphism.

To prove that $g$ is an isomorphism, we first note that $A/A_{\text{tors}}$ is a $\Q$-vector space. Then we can use Milne's deconstruction theorem for reductive groups, \cite[Proposition 19.27,Proposition 19.29]{Milne_AG}. 

First we prove that $g$ is injective. Let $\alpha_1,\alpha_2\in \text{Hom}(G^\circ, \mathbb{G}_m)$, such that $g(\alpha_1)$ and $g(\alpha_2)$ are the same element in $\text{Hom}(Z^\circ, \mathbb{G}_m)$. Denote the image by $\beta$. Since the target of $\alpha_1$, $\alpha_2$ is reductive, $\alpha_1,\alpha_2$ factor through some connected reductive groups $G_1,G_2$. Denote by $Z_1,Z_2$ the center of $G_1,G_2$, and by $Z_1^{\circ},Z_2^{\circ}$ the identity component of $Z_1,Z_2$, respectively. Since $g(\alpha_1)=g(\alpha_2)$, there is a map $\phi:Z_1^{\circ}\to Z_2^{\circ}$ in the projective system such that $g(\alpha_1)=g(\alpha_2)\circ \phi$. Since $Z$ is the projective limit of the centers of the components of $G^{\circ}$, there is a map $\epsilon: G_1\to G_2$ such that the following diagram commutes:
\begin{center}
\begin{tikzcd}
Z_1^{\circ}\arrow[hookrightarrow]{r}\arrow[d,"\phi"]& G_1\arrow[d,"\epsilon"]\\
Z_2^{\circ}\arrow[hookrightarrow]{r}& G_2\\
\end{tikzcd}
\end{center}

Moreover, the following diagram commutes as well:
\begin{center}
\begin{tikzcd}
G_1\arrow[r,"Ab"]\arrow[d,"\epsilon"]& Z_1^{\circ}\arrow[d,]\\
G_2\arrow[r,"Ab"]& Z_2^{\circ}\\
\end{tikzcd}
\end{center}

Therefore, the composition of the above two diagrams also commutes:
\begin{center}
\begin{tikzcd}
Z_1^{\circ}\arrow[r,]\arrow[d,"\phi"]& Z_1^{\circ}\arrow[d,]\\
Z_2^{\circ}\arrow[r,]& Z_2^{\circ}\\
\end{tikzcd}
\end{center}

By the deconstruction theorem mentioned above, combined with the fact that $G$ is defined over $\overline{\Q_{\lambda}}$, and thus split, the composition of the injection $Z_i\to G_i$ with the abelianization $G_i\to Z_i^{\circ}$, is actually an isogeny. In particular, it is surjective. Since $\mathbb{G}_m$ is commutative, $\alpha_1,\alpha_2$ further factor through $Z_1^{\circ},Z_2^{\circ}$. Thus, by the commutativity of the above diagram and that $g(\alpha_1)=g(\alpha_2)$, $\alpha_1$ and $\alpha_2$ are indeed the same, proving the injectivity of $g$. 

To prove that $g$ is surjective, we first prove that the cokernel of $g$ is torsion. Let $\beta\in \text{Hom}(Z^\circ, \mathbb{G}_m)$. Then $\beta$ factors through some center of $H$, where $H$ is a component of $G^{\circ}$. By the above claim that the composition of the injection $Z_i\to G_i$ with the abelianization $G_i\to Z_i^{\circ}$, is actually a power map, some power of $\beta$ must be in the image of $g$. This proves that the cokernel of $g$ is torsion. However, since $\text{Hom}(G^\circ, \mathbb{G}_m)$ is a $\Q$-vector space, the cokernel is indeed trivial, which proves the surjectivity of $g$. Thus, $g$ is an isomorphism. 
\end{proof}

The following proposition is the key to identify those representations which can be in some compatible system.

\begin{proposition}
\label{proposition 5.7}
    Let $\mathcal{X}$ be a smooth geometrically integral algebraic stack of finite presentation over a finite field $\kappa$ of characteristic p and let $x$ be a geometric point. Let $l$ be a prime number not dividing $p$, and let $\rho:\pi_1(\mathcal{X},x)\to G(\overline{\Q_{\lambda}})$ be an irreducible representation with Zariski dense image. Then the following are equivalent:
    \begin{enumerate}
        \item[(1)] $\rho$ is plain of characteristic $p$.
        \item[(2)] there exists a faithful representation $\theta$ of $G$, with a finite extension $\kappa'/\kappa$, and an $x\in \mathcal{X}(\kappa')$, such that the Frobenius eigenvalues of $x$ in $\mathrm{det}\circ \theta\circ \rho$ are plain of characteristic $p$.
    \end{enumerate}
\end{proposition}
\begin{proof}
    It is clear that $(1)$ implies $(2)$. For the converse, let $\theta$ be the faithful representation of $G$. Write $\chi:=\mathrm{det}\circ \theta\circ \rho$. By Lemma \ref{lemma 5.5}, there exists an $n$ such that $\chi^n$ factors through $\pi_1(\mathcal{X},x)\twoheadrightarrow \mathrm{Gal}(\overline{\kappa}/\kappa)$. Let $\beta$ be the image of the geometric Frobenius $\mathrm{Frob}_x$ under $\chi^n$ of $x$. Suppose the dimension of the representation $\theta\circ \rho$ is $d$, and let $\epsilon$ be a rank one representation sending $\mathrm{Frob}_x$ to $\beta^{\frac{1}{dn}}$, then $(\theta\circ \rho)\bigotimes \epsilon$ is an irreducible representation with finite order determinant into the value group $\mathrm{GL}_n$. By \cite[theorem 0.1(1)]{Zheng_2018}, all the Frobenius eigenvalues of $(\theta\circ \rho)\bigotimes \epsilon$ are $p$-Weil numbers of weight $0$, and in particular, they are plain of characteristic $p$ in the sense of Definition \ref{definition 5.1}. Thus, the Frobenius eigenvalues of $\theta\circ \rho$ are also plain of characteristic $p$. By Proposition \ref{proposition 5.2}, the implication $(2)\Longrightarrow (1)$ is proved.
\end{proof}

\begin{proposition}
\label{proposition 5.8}
A finite-dimensional representation of $\hat{\Pi}^{\mathrm{red}}_{\lambda}$ is a representation of $\hat{\Pi}^{\text{plain of char }p}_{\lambda}$ if and only if the corresponding $\lambda$-adic local system (with value group $\mathrm{GL}_n$) on $X$ is plain of char $p$.
\end{proposition}
\begin{proof}
By definition, $\hat{\Pi}^{\text{plain of char }p}$ is the fiber product of $\hat{\Pi}$ and $\mathcal{A}$. A representation factors through $\hat{\Pi}^{\text{plain of char }p}$ if and only if its determinant (of the corresponding local system) takes $\mathrm{Frob}_x$ to elements in $\mathrm{Plain}_p$ for all closed point $x$. By Proposition \ref{proposition 5.7}, this condition is equivalent to the associated local system being plain of characteristic $p$.

\end{proof}

By construction, $\pi_0(\hat{\Pi}^{\text{plain of char }p}_{(\lambda)})=\pi_0(\hat{\Pi}_{(\lambda)})=\Pi$. 
Moreover, the Frobenius eigenvalues of a plain of characteristic $p$ represetation will be inside $\overline{\Q}$, so we can construct a similar diagram as (Diagram \ref{diag 5}):
\begin{align}
\label{diag 5}
    \Pi_{\mathrm{Fr}}\to [\hat{\Pi}^{\text{plain of char }p}](\overline{\Q})\to \Pi
\end{align}
\begin{theorem}
\label{theorem 5.9}
Let $\mathcal{X}$ be a smooth algebraic stack of finite presentation over $\mathbb{F}_q$. Let $\lambda$ and $\lambda'$ be non-Archimedean places of $\overline{\Q}$ not dividing $p$. Then there exists a unique isomorphism $\hat{\Pi}^{\text{plain of char }p}_{(\lambda)}\xrightarrow{\sim} \hat{\Pi}^{\text{plain of char }p}_{(\lambda')}$ in the category $\mathbf{Pro\text{-}red}(\overline{\Q})$ which sends the diagram $\Pi_{\mathrm{Fr}} \to [\hat{\Pi}^{\text{plain of char }p}_{(\lambda)}](\overline{\Q})\to \Pi$ to a similar diagram $\Pi_{\mathrm{Fr}} \to [\hat{\Pi}^{\text{plain of char }p}_{(\lambda')}](\overline{\Q})\to \Pi$.    
\end{theorem}
\begin{proof}
    By definition of $\hat{\Pi}^{\text{plain of char }p}_{\lambda}$, it is defined as a fiber product of $\hat{\Pi}_{\lambda}$ and $\mathcal{A}_{\lambda}$ over $\hat{Z}$, where we recall that $\mathcal{A}_{\lambda}=\operatorname{Hom}(\text{Plain}_p,\mathbb{G}_m)\bigotimes_{\overline{\Q}}\overline{\Q_{\lambda}}$. Thus, by Theorem \ref{theorem 4.2}, and the above Proposition \ref{proposition 5.8}, the existence of the above isomorphism is proved.

    For the uniqueness, we observe that the image of $\Pi_{\mathrm{Fr}}$ is Zariski dense in $\mathcal{A}$. Therefore, combined with the uniqueness result of Theorem \ref{theorem 4.2}, we obtained the uniqueness here.
    
\end{proof}
\begin{proof}[Proof of Theorem \ref{theorem 5.4}]
Let $G$ be a reductive algebraic $\Q$-group. An irreducible representation $\rho_{\lambda}:\Pi\to G(\overline{\Q_{\lambda}})$ such that for some $x$, the eigenvalues of $\rho_{\lambda}^{ab}(\mathrm{Frob}_x)$ are plain of characteristic $p$, is equivalent to a plain of characteristic $p$ irreducible representation, by Proposition \ref{proposition 5.7}.

Thus, it is the same as a representation of $\hat{\Pi}_{\lambda}^{\text{plain of char }p}$ into $\mathrm{GL}_{n,\overline{\Q_{\lambda}}}$. By the equivalence of categories: $\mathbf{Pro\text{-}red}(\overline{\Q})\xrightarrow{\sim} \mathbf{Pro\text{-}red}(\overline{\Q}_\lambda)$, we obtain a representation of $\hat{\Pi}_{(\lambda)}^{\text{plain of char }p}$ into $\mathrm{GL}_{n,\overline{\Q}}$. $\Psi^{-1}$, the inverse of the unique isomorphism, induces a representation of $\hat{\Pi}_{(\lambda')}^{\text{plain of char }p}$ into $\mathrm{GL}_{n,\overline{\Q}}$. Again by the equivalence of categories: $\mathbf{Pro\text{-}red}(\overline{\Q})\xrightarrow{\sim} \mathbf{Pro\text{-}red}(\overline{\Q}_{\lambda'})$, we obtain a representation of $\rho_{\lambda'}:\hat{\Pi}_{\lambda'}^{\text{plain of char }p}$ into $\mathrm{GL}_{n,\overline{\Q_{\lambda'}}}$. Since the $\Psi$ respect the Diagram \ref{diag 1} and its $\lambda'$ counterpart, $\rho_{\lambda}$ and $\rho_{\lambda'}$ are compatible and $\rho_{\lambda}^{ab}$ and $\rho_{\lambda'}^{ab}$ are also compatible. 
\end{proof}

\section{Application}
Here, the major objective is to apply Theorem \ref{theorem 4.2} and Theorem \ref{theorem 5.9} to prove the compatibility of the canonical $l$-adic local system on Shimura stacks (see Theorem \ref{theorem 6.25}, Theorem \ref{theorem 6.29}).
\subsection{The adjoint case}
We first restrict ourselves to the case where the canonical local system takes values in an adjoint group. 
\subsubsection{Definition of the Shimura stacks and Construction of the Canonical $l$-adic local system on Shimura stacks}
\label{Sec 6.1.1}
\paragraph{Definition of the Shimura stacks}
Let $(G,X)$ be a Shimura datum satisfying axiom SV5 as in \cite[Chapter 5]{Milne_SVI}. Let $K$ be a possibly non-neat compact open subgroup of $G$. Let $E = E(G, X)$ be the reflex field of $(G, X)$. The theory of canonical models (\cite{Deligne_Shimura}, \cite{Milne_SV}) produces an inverse system $\{\mathrm{Sh}_K\}_K$, indexed over neat compact open subgroups of $G(\mathbb{A}_f)$, of smooth quasi-projective varieties over $E$.
\begin{lemma}
\label{lemma 6.1}
For the reductive group $G$, let $G(\R)_{+}$ be the kernel of the following map: $G(\R)\xrightarrow{\mathrm{ad}(\R)}G^{\mathrm{ad}}(\R)\to \pi_0(G^{\mathrm{ad}}(\R))$. Let $G(\Q)_{+}=G(\Q)\cap G(\R)_{+}$. For two neat normal open subgroups $K_0,K_1$ of finite index in $K$, let $K_2=K_0\cap K_1$, then viewed as DM stacks over the reflex field $E:=E(G,X)$, $[\mathrm{Sh}_{K_i}(G,X)/(K/K_i)]\cong [\mathrm{Sh}_{K_2}(G,X)/(K/K_2)]$ canonically for $i=0,1$. Also, $[\mathrm{Sh}_{K_i}(G,X)/(K/K_i)](\C)\cong \bigsqcup_{g\in \mathcal{C}}[\Gamma_g \backslash X^{+}]$, where $X^+$ is a connected component of $X$, $\mathcal{C}$ is a set of representatives for the double coset space $G(\Q)_{+}\backslash G(\mathbb{A}_f)/K$ and $\Gamma_g=gKg^{-1}\cap G(\Q)_{+}$.
\end{lemma}
\begin{proof}
First we identify the two quotient stacks. Without loss of generality I can assume that $K_0\subseteq K_1$. By \cite[III, Corollaire 2.5]{algebraic_groups}, since $K_1/K_0$ act freely on $G(\Q)\backslash X\times G(\mathbb{A}_f)/K_0=\mathrm{Sh}_{K_0}(G,X)(\C)$, $K_1/K_0$ act freely on $\mathrm{Sh}_{K_0}(G,X)$. Since $\mathrm{Sh}_{K_0}(G,X)$ is quasi-projective (it has a good compactification), any finite set of points is contained in an affine open subscheme, so the quotient algebraic space $\mathrm{Sh}_{K_0}(G,X)/(K_1/K_0)$ is actually a scheme. Then $\mathrm{Sh}_{K_0}(G,X)/(K_1/K_0)\cong \mathrm{Sh}_{K_1}(G,X)$ by \cite[Corollary 3.16]{Milne_SVI}. Therefore for a $E$-scheme $T$, $(\mathrm{Sh}_{K_0}(G,X)/(K_1/K_0))(T)=\mathrm{Sh}_{K_0}(G,X)(T)/(K_1/K_0)=\mathrm{Sh}_{K_1}(G,X)(T)$.

The canonical projection $\mathrm{Sh}_{K_0}(G,X)/(K_1/K_0)\to \mathrm{Sh}_{K_1}(G,X)$ induced the canonical morphism $[\mathrm{Sh}_{K_0}(G,X)/(K/K_0)]^{pre}\to [\mathrm{Sh}_{K_1}(G,X)/(K/K_1)]^{pre}$, where the superscript "pre" means the quotient prestack. We prove that this canonical morphism is an isomorphism. For a scheme $T$ over the reflex field $E$, by definition, $[\mathrm{Sh}_{K_i}(G,X)/(K/K_i)]^{pre}(T)$ is the groupoid $[\mathrm{Sh}_{K_i}(G,X)(T)/(K/K_i)]$, for $i=0,1$. The object of the groupoid, say for $i=0$, are  the elements in $\mathrm{Sh}_{K_0}(G,X)(T)$, and a morphism $a\to b$ is an element $g\in K/K_0$ such that $b=g.a$. Since $(\mathrm{Sh}_{K_0}(G,X)/(K_1/K_0))(T)=\mathrm{Sh}_{K_1}(G,X)(T)$, the canonical morphism induces an equivalence of groupoids, and this equivalence is natural in $T$. Therefore, the two quotient prestacks are isomorphic. By the universal property of stackification, the two quotient stacks are therefore also isomorphic.

To prove the second statement (for $i=0$, say, the other is similar), first note that since $\C$ is algebraically closed, the only connected \'etale cover of itself is trivial. So $[\mathrm{Sh}_{K_0}(G,X)/(K/K_0)](\C)=[\mathrm{Sh}_{K_0}(G,X)/(K/K_0)]^{pre}(\C)=[\mathrm{Sh}_{K_0}(G,X)(\C)/(K/K_0)]$. There is a map:

\begin{align*} 
    \phi:[G(\Q)\backslash X\times G(\mathbb{A}_f)/K]\longmapsto[\mathrm{Sh}_{K_0}(G,X)(\C)/(K/K_0)].
\end{align*}
On the level of object, it maps $(x,g)$ to $\overline{(x,g)}$, the equivalence class of $(x,g)$ in $\mathrm{Sh}_{K_0}(G,X)(\C)=G(\Q)\backslash X\times G(\mathbb{A}_f)/K_0$. On the level of morphisms, it maps $(q,k)$ to $kK_0$. Let me show that this map is an isomorphism of groupoids.

Since both the source and target are groupoids, to prove the full faithfulness, I only need to show that the stabilizers of an object in the source is mapped bijectively to the stabilizers of its image in the target. 

For the faithfulness, fix an object $(x,g)\in [G(\Q)\backslash X\times G(\mathbb{A}_f)/K]$. Suppose that for some $(q_1,k_1),(q_2,k_2)$ which stabilize $(x,g)$, $k_1K_0=k_2K_0$. Since the natural map $G(\Q)_{+}\backslash X^{+}\times G(\mathbb{A}_f)\to G(\Q)\backslash X\times G(\mathbb{A}_f)$ is an isomorphism, I may assume that $x\in X^{+}$ and that $q_1,q_2\in G(\Q)_{+}$. In this case, $(q_2^{-1}q_1,k_1k_2^{-1})$ also stabilizes $(x,g)$. Thus $q_2^{-1}q_1\in G(\Q)_{+}\cap gK_0g^{-1}$. Denote the later group by $\lambda_g$, it is a congruence subgroup in $G(\Q)$ with arithmetic image in $G^{ad}(\Q)\subseteq G^{ad}(\R)$. Since $G^{ad}(\R)$ is a semisimple real Lie group with trivial center, it admits a faithful finite dimensional representation. So by \cite[proposition 3.6]{Milne_SVI}, $\lambda_g$ is a discrete subgroup of $G^{ad}(\R)$. Therefore, its image, which I denote as $\Lambda_g$, in $\mathrm{Aut}(X^+)$ is also discrete. Since $K_0$ is neat, $\Lambda_g$ is also neat, and thus torsion-free. So by \cite[proposition 3.1]{Milne_SVI}, $\Lambda_g$ acts freely on $X^{+}$, and hence $q_2=q_1$. As a result, $k_1=k_2$ as well. This proves the faithfulness of $\phi$.

To prove the fullness of $\phi$, suppose that some $k\in K$ stabilizes the equivalence class of $(x,g)$. Then there is a $q\in G(\Q)$, $k\in K_0$ such that $(x,gk^{-1})=(qx,qgk_0)$. Then $(q,k_0k)$ stabilizes $(x,g)$ and $k_0kK_0=kK_0$ since $K_0$ is normal. This proves the fullness of $\phi$.

The essential surjectivity is clear, since for any equivalence class of $(x,g)\in \mathrm{Sh}_{K_0}(G,X)(\C)$, $(x,g)\in [G(\Q)\backslash X\times G(\mathbb{A}_f)/K]$ will be mapped to that equivalence class. Therefore, $\phi$ is an isomorphism. 

Now, there is another map 
\begin{align*}
    \psi:\bigsqcup_{\mathcal{C}}[\Gamma_g\backslash X^{+}]\longmapsto[G(\Q)_{+}\backslash X^{+}\times G(\mathbb{A}_f)/K].
\end{align*} On the level of objects, given an object $x$ in $[\Gamma_g\backslash X^{+}]$ for some $g$, $\psi$ maps $x$ to $(x,g)$ for this specific $g$. On the level of morphisms, given $q\in gKg^{-1}\cap G(\Q)_{+}$, there is a $k\in K$ such that $q=gkg^{-1}$. $\psi$ maps $q$ to $(q,k)$.

The faithfulness is clear by the description of $\psi$. To prove the fullness, if $(q,k)$ stabilizes $(x,g)$, then $qgk=g$, so $q$ stabilizes $x$ and belongs to $\Gamma_g$.

To prove the essential surjectivity, for an object $(x,h)$ in $[G(\Q)_{+}\backslash X^{+}\times G(\mathbb{A}_f)/K]$, there is some $q\in G(\Q)_{+},g\in \mathcal{C},k\in K$ such that $h=qgk$. Then $(x,h)$ is isomorphic to the object $(q^{-1}x,g)$, which is the image of $q^{-1}x$ in $[\Gamma_g\backslash X^{+}]$. Thus, $\psi$ is also essentially surjective and thus an isomorphism.

Therefore, $[\mathrm{Sh}_{K_i}(G,X)/(K/K_i)](\C)\cong \bigsqcup_{g\in \mathcal{C}}[\Gamma_g \backslash X^{+}]$, proving the second claim of this lemma.

\end{proof}
\begin{definition}
    For a compact open subgroup $K$ of $G(\mathbb{A}_f)$, define $\mathrm{Sh}_{K}[G,X]:=[\mathrm{Sh}_{K_0}(G,X)/(K/K_0)]$ for an open neat subgroup $K_0$ in $K$ of finite index.
\end{definition}
\begin{lemma}
\label{lemma 6.3}
    Let $V_1, V_2, \dots, V_n$ denotes the geometrically connected components of $\mathrm{Sh}_{K_0}(G,X)$. Then $K/K_0$ acts on $\{V_1,V_2,...,V_n\}$ transitively. The geometrically connected components of $\mathrm{Sh}_{K}[G,X]$ are $[\bigsqcup_{i \in \mathcal{I}}V_i/(K/K_0)]$, where $\mathcal{I}$ is an orbit in $\{V_1,V_2,...,V_n\}$ under the action of $K/K_0$, and that $[\bigsqcup_{i \in \mathcal{I}}V_i/(K/K_0)](\C)=[\Gamma_g\backslash X^{+}]$ for some $g$ in $\mathcal{C}$. Moreover, let $\mathcal{I}_1$ be the orbit for $V_1$, and let $H$ be the stabilizer of $V_1$ in $K/K_0$, then $[\bigsqcup_{i \in \mathcal{I}_1}V_i/(K/K_0)]\cong [V_1/H]$.
\end{lemma}
\begin{proof}
    By \cite[Theorem 7.2]{Shimura's_conjecture} and \cite[Theorem 13.6]{Milne_SVI}, $K/K_0$ can be viewed as a constant group scheme. Since the $V_i$'s are all geometrically connected, it suffices to prove that the stack $[\bigsqcup_{i \in \mathcal{I}}V_i/(K/K_0)]$ is connected whenever the $V_i$'s are in the same orbit. Let $U$ be a connected component of $[\bigsqcup_{i \in \mathcal{I}}V_i/(K/K_0)]$. Since $\bigsqcup_{i \in \mathcal{I}}V_i\to [\bigsqcup_{i \in \mathcal{I}}V_i/(K/K_0)]$ is a surjection, there must be some $i$ such that the image of $V_i$ in $[\bigsqcup_{i \in \mathcal{I}}V_i/(K/K_0)]$ has a nonempty intersection with $U$. Since $U$ is a connected component, and that $V_i$ is connected, $U$ must contain the image of $V_i$. However, since all $V_i$'s are in the same orbit, $U$ has an nonempty intersection with the image of $V_j$ in $[\bigsqcup_{i \in \mathcal{I}}V_i/(K/K_0)]$ for all other $j\ne i$. Thus, $U$ contains the image of $V_i$ for all $i\in \mathcal{I}$, i.e., $U=[\bigsqcup_{i \in \mathcal{I}}V_i/(K/K_0)]$, so $[\bigsqcup_{i \in \mathcal{I}}V_i/(K/K_0)]$ is connected, proving the first statement. 

    For the second statement, $[\bigsqcup_{i \in \mathcal{I}}V_i/(K/K_0)](\C)=[\bigsqcup_{i \in \mathcal{I}}V_i/(K/K_0)]^{pre}(\C)=[\bigsqcup_{i \in \mathcal{I}}V_i(\C)/(K/K_0)]$, where the last one is regarded as a groupoid. Denote $G(\Q)_{+}\backslash G(\mathbb{A}_f)/K_0$ by $\mathcal{C}_0$, then each $V_i$ corresponds to some element in $\mathcal{C}_0$, say $h_i$. $K/K_0$ acts on $\mathcal{C}_0$ by right translation. If $V_i$ and $V_j$ are in the same orbit under the action of $K/K_0$, then $h_i,h_j$ are in the same orbit under the action of $K/K_0$. Denote this orbit, which is also an element in $\mathcal{C}=G(\Q)_{+}\backslash G(\mathbb{A}_f)/K$, by $g$, then $[\bigsqcup_{i \in \mathcal{I}}V_i(\C)/(K/K_0)]=[\bigsqcup_{i}(\Gamma_{h_i}\backslash X^{+})/(K/K_0)]=[\Gamma_g\backslash X^{+}]$, proving the second statement.

    To prove the last statement, by the universal property of stackification, suffices to prove that the canonical morphism $f:[V_1/H]^{pre}\cong [\bigsqcup_{i \in \mathcal{I}_1}V_i/(K/K_0)]^{pre}$ of prestacks induced by the inclusion of $V_1\to $ is an isomorphism. The fully faithfulness of $f$ is clear since $H$ is the stabilizer of $V_1$, and the essential surjectivity is also clear since $\mathcal{I}_1$ is the orbit of $V_1$. Therefore, $f$ is an isomorphism, which proves the last statement.

\end{proof}

\paragraph{Construction of the Canonical $l$-adic local system on Shimura stacks}

Fix a point $s \in \mathrm{Sh}(\C)$. For a neat open compact subgroup $K_0$ of $K$, let $S_{K_0,s}$ be the geometrically connected component of $\mathrm{Sh}_{K_0}(G,X)$ containing $sK_0$, and let $H_{K_0}$ be the stabilizer of $S_{K_0,s}$ in $K/K_0$. Define $[S_{K,s}]=[S_{K_0,s}/H_{K_0}]$, it is the geometrically connected component of $\mathrm{Sh}_K[G,X]$ containing the image of $s$. Since $K_0$ is an arbitrary compact open neat subgroup in $K$, and that $\cap_{K_0}E_{K_0,s} = E_{K,s}$ for $K_0$ ranging over compact open neat subgroup in $K$, $[S_{K,s}]$ is unambiguously defined and defined over $E_{K,s}$. To simplify notation, in the \'etale fundamental groups we will simply write $\overline{s}$ for any of these geometric points $sK$ (sometimes with a subscript to emphasize field of definition for $\overline{s})$.

By the definition of the Shimura stacks, for a chosen neat open subgroup $K_0$ of finite index in $K$, $\mathrm{Sh}_{K_0}(G,X)$ is a finite \'etale cover of $[\mathrm{Sh}_{K_0}(G,X)/(K/K_0)]$. Let $K'$ be an open normal subgroup in $K$. and let $K''=K'\cap K_0$, then $K''$ is neat, open and normal in $K$. So by Lemma \ref{lemma 6.1}, $\mathrm{Sh}_{K}[G,X]\cong [\mathrm{Sh}_{K''}(G,X)/(K/K'')]$ and $\mathrm{Sh}_{K'}[G,X]\cong [\mathrm{Sh}_{K''}(G,X)/(K'/K'')]$. Thus the Shimura stack $\mathrm{Sh}_{K'}[G,X]$ is a finite Galois cover of $\mathrm{Sh}_{K}[G,X]$ with Galois group $K/K'$, and $[S_{K',s}]/[S_{K,s}]$ is also a Galois cover with Galois group a subgroup of $K/K'$. Call the field of definition for $[S_{K',s}]$ as $E_{K',s}$. 

The above gives a representation:
\begin{align*}
    \pi_1([S_{K,s}],\overline{s})\to K/K''
\end{align*}

Taking the limit over all the open compact subgroup $K'$ of $K$, there is a representation:
\begin{align*}
    \rho_{K,s}:\pi_1([S_{K,s}],\overline{s})\to K
\end{align*}
As $K\in G(\mathbb{A}_f)$, for all prime $l$, there is a projection:
\begin{align*}
    \rho_{K,s,l}:\pi_1([S_{K,s}],\overline{s})\to K\hookrightarrow G(\mathbb{A}_f)\twoheadrightarrow G(\Q_\ell) 
\end{align*}

\begin{lemma}
\label{lemma 6.4}
    Let $\mathcal{X}$ be a connected normal algebraic stack of finite presentation over an algebraically closed characteristic $0$ field $F$, and let $\Omega$ be a field extension of $F$ such that $\Omega$ is also an algebraically closed field of characteristic $0$. Let $x$ be an $F$ point of $\mathcal{X}$, and denote by $x_{\Omega}$ the pullback of $x$ to $\mathcal{X}_{\Omega}$, then the following canonical map $\pi_1(\mathcal{X}_{\Omega},x_{\Omega})\to \pi_1(\mathcal{X},x)$ is an isomorphism.
\end{lemma}
\begin{proof}
To prove that the canonical morphism is an isomorphism, it suffices to prove that the functor $F: F\acute{E}t(\mathcal{X}) \to F\acute{E}t(\mathcal{X}_{\Omega})$ from the category of finite \'etale covers of $\mathcal{X}$ to the category of finite \'etale covers of $\mathcal{X}_{\Omega}$ is fully faithful and essentially surjective. 

We prove the fully faithfulness by the following lemma:

\begin{lemma}[Descending of morphisms of finite \'etale covers of algebraic stacks]
\label{lemma 6.5}
Let $\mathcal{Z}$ be an algebraic stack of finite presentation over an algebraically closed field $k$, let $\mathcal{E}_1, \mathcal{E}_2$ be two finite \'etale schemes over $\mathcal{Z}$, and let $K/k$ be a field extension. The natural map
\[
\operatorname{Hom}_{\mathcal{Z}}(\mathcal{E}_1, \mathcal{E}_2)
\;\longrightarrow\;
\operatorname{Hom}_{\mathcal{Z}_K}((\mathcal{E}_1)_K, (\mathcal{E}_2)_K)
\]
is bijective.
\end{lemma}

\begin{proof}
Since $\mathcal{Z}$ is Noetherian, we may assume that $\mathcal{E}_1$ is connected (as $\mathcal{E}_1$ is also Noetherian, the connected components in $\mathcal{E}_1$ are open). By the universal property of pullbacks, the set of morphisms $\mathcal{E}_1 \to \mathcal{E}_2$ over $\mathcal{Z}$ is in bijection with the set of sections of the \'etale morphism (the canonical projection)
\[
\mathcal{E}_1 \times_\mathcal{Z} \mathcal{E}_2 \longrightarrow \mathcal{E}_1
\]
over $\mathcal{Z}$. The latter set is in bijection with
\[
\pi_0(\mathcal{E}_1 \times_{\mathcal{Z}} \mathcal{E}_2).
\]
Indeed, since the projection
\[
\mathcal{E}_1 \times_{\mathcal{Z}} \mathcal{E}_2 \longrightarrow \mathcal{E}_1
\]
is finite \'etale, any section is itself finite \'etale (by cancellation properties for morphisms of algebraic stacks), hence an open and closed immersion. Its image therefore defines a connected component of
\(\mathcal{E}_1 \times_{\mathcal{Z}} \mathcal{E}_2\).

Similarly, since $\mathcal{Z}_K$ is also Noetherian, the same argument applies over $\mathcal{Z}_K$. Moreover, the natural map
\[
\pi_0((\mathcal{E}_1)_K \times_{\mathcal{Z}_K} (\mathcal{E}_2)_K)
\;\longrightarrow\;
\pi_0(\mathcal{E}_1 \times_\mathcal{Z} \mathcal{E}_2)
\]
is bijective. This completes the proof.
\end{proof}

Next we proof the essential surjectivity by the following lemmas:
\begin{lemma}
\label{lemma 6.6}
    The notations are the same as in Lemma \ref{lemma 6.4}. If $\mathcal{X}$ is an algebraic space, the functor $F: F\acute{E}t(\mathcal{X}) \to F\acute{E}t(\mathcal{X}_{\Omega})$ is essentially surjective.
\end{lemma}
\begin{proof}
    We use presentations of algebraic spaces by schemes to prove the lemma. Write $\mathcal{X}_F$ for $\mathcal{X}$ to emphasize that it is originally defined on $F$.
    Suppose $\pi: \mathcal{E}_{\Omega}\to \mathcal{X}_{\Omega}$ is a finite \'etale covering. Let $U_F\to \mathcal{X}_F$ be an \'etale surjection with \(U_F\) a scheme of finite presentation over $F$, and let $R_F: = U_F\times_{\mathcal{X}_F}U_F\overset{\mathrm{pr}_1}{\underset{\mathrm{pr}_2}\rightrightarrows} U_F$ be the relation scheme. Then $U_{\Omega}\to \mathcal{X}_{\Omega}$ is an \'etale surjection and $V_{\Omega}:=U_{\Omega}\times_{\mathcal{X}_{\Omega}}\mathcal{E}_{\Omega}\to U_{\Omega}$, as the pullback of the representable and finite \'etale morphism $\pi: \mathcal{E}_{\Omega}\to \mathcal{X}_{\Omega}$, is a finite and \'etale morphism between schemes. By \cite[Corollary 2.10]{ShuNotes}, $V_{\Omega}\to U_{\Omega}$ descends to $F$, and we pull $V_F\to U_F$ back along $\mathrm{pr}_1$, $\mathrm{pr}_2$ to obtain two $R_F$-schemes:
    \[
    \mathrm{pr}_1^*V_F \;=\; V_F\times_{U_F,\mathrm{pr}_1} R_F,
    \qquad
    \mathrm{pr}_2^*V_F \;=\; V_F\times_{U_F,\mathrm{pr}_2} R_F.
    \]
    Over $\Omega$ the two pullbacks of $V_{\Omega}$ along $R_{\Omega}\overset{\mathrm{pr}_1}{\underset{\mathrm{pr}_2}\rightrightarrows} U_{\Omega}$ are canonically isomorphic, since $V_{\Omega}$ was obtained by pulling $\mathcal{E}_{\Omega}$ back along $U_{\Omega}\to\mathcal X_{\Omega}$.
    By \cite[Corollary 2.10]{ShuNotes} again, that canonical isomorphism and its cocycle condition
    are defined over $F$; we obtain an isomorphism of $R_F$-schemes
    \[
    \varphi_F:\mathrm{pr}_1^*V_F \xrightarrow{\sim} \mathrm{pr}_2^*V_F
    \]
    which satisfies the cocycle condition on $U_F\times_{\mathcal X_{F}} U_F \times_{\mathcal X_{F}} U_F$.
    Thus we are able to form the relation scheme \(R_{V,F}\) on \(V_F\).
    Using $\varphi_F$ we define a relation on $V_F$ by setting
    \[
    R_{V,F} \;:=\; \mathrm{pr}_1^*V_F \;\cong_{\varphi_F}\; \mathrm{pr}_2^*V_F
    \qquad\text{with the two structure maps } s,t: R_{V,F}\rightrightarrows V_F
    \]
    given by the projections to the two factors of $\mathrm{pr}_i^*V_F\cong V_F\times_{U_F}R_F$.

    Now we are able to construct the descent $\mathcal{E}_{F}$. Define the presheaf (on the big fppf site over $\mathcal{X}_{F}$)
    \[
    \mathcal{F} \;=\; V_F / R_{V,F},
    \]
    Let $\mathcal{E}_{F}$ denote the fppf sheafification of $\mathcal F$.
    
    Because $R_F\overset{\mathrm{pr}_1}{\underset{\mathrm{pr}_2}{\rightrightarrows}} U_F$ is an \'etale equivalence relation, the relation $R_{V,F}\rightrightarrows V_F$ is also an \'etale equivalence relation, as it is the pullback of $R_F\overset{\mathrm{pr}_1}{\underset{\mathrm{pr}_2}{\rightrightarrows}} U_F$ along $V_F\to U_F$. It is clear that the fppf sheafification is representable
    by an algebraic space; thus $\mathcal{E}_{F}$ is an algebraic space and the canonical map
    \[
    q:V_F \to \mathcal{E}_{F}
    \]
    is an \'etale surjection presenting $\mathcal{E}_{F}$ as the quotient of $V_F$ by $R_{V,F}$.
    
    Finally we construct the map \(\mathcal E_{F}\to\mathcal X_{F}\) and check its properties. By our construction above, the map $V_F \xrightarrow{} U_F$ descends to a map of algebraic spaces
    \[
    \pi_F: \mathcal{E}_{F} \longrightarrow \mathcal X_{F}.
    \]
    Locally on $\mathcal X_{F}$ (pull back along the \'etale cover $U_F\to\mathcal X_{F}$) this map is identified with $V_F\to U_F$,
    which is finite \'etale; therefore $\pi_F$ is representable and finite \'etale. Finally, by construction all objects and identifications were chosen so that
    base change along $F\hookrightarrow \Omega$ recovers the original $V_{\Omega}\to U_{\Omega}$, the original relation on it, and hence the original cover $\mathcal{E}_{\Omega}$ and morphism $\pi$. Thus the lemma is proved.    
\end{proof}

\begin{lemma}
\label{lemma 6.7}
    More generally, if $\mathcal{X}$ is an algebraic stack, the functor $F: F\acute{E}t(\mathcal{X}) \to F\acute{E}t(\mathcal{X}_{\Omega})$ is also essentially surjective.
\end{lemma}
\begin{proof}
    We use presentations of algebraic stacks by algebraic spaces to prove the lemma. Suppose $\pi: \mathcal{E}_{\Omega}\to \mathcal{X}_{\Omega}$ is a finite \'etale covering. We first choose a smooth presentation for \(\mathcal X_F\):
    pick a smooth surjection \(p:U_F\to\mathcal{X}_F\) with \(U_F\) a scheme of finite presentation over $F$. 
    Let 
    \[
    R_F:=U_F\times_{\mathcal{X}_F}U_F\overset{\mathrm{pr}_1}{\underset{\mathrm{pr}_2}\rightrightarrows} U_F
    \]
    be the associated relation algebraic space with projection maps $\mathrm{pr}_1,\mathrm{pr}_2$. 
    
    \bigskip
    
    Then we will produce a cover and a relation algebraic space whose quotient will be our desired descent for $\mathcal{E}_{\Omega}$.
    Form the fiber product
    \[
    V_{\Omega} := \mathcal{E}_{\Omega}\times_{\mathcal{X}_{\Omega}} U_{\Omega}.
    \]
    Because $\pi$ is representable and finite \'etale, the map $V_{\Omega}\to U_{\Omega}$ is a finite \'etale morphism of algebraic spaces of finite presentation. By the previous Lemma \ref{lemma 6.6}, $V_{\Omega}\to U_{\Omega}$ descends to $F$: $V_F\to U_F$.
    
    For the relation algebraic space, let $R_{\Omega} := R_F\otimes_F \Omega = U_{\Omega}\times_{\mathcal{X}_{\Omega}}U_{\Omega} \rightrightarrows U_{\Omega}$ with projections $\mathrm{pr}_1,\mathrm{pr}_2:R_{\Omega}\to U_{\Omega}$.
    Pull back \(V_{\Omega}\to U_{\Omega}\) along these projections to obtain two $R_{\Omega}$-spaces
    \[
    \mathrm{pr}_1^*V_{\Omega},\qquad \mathrm{pr}_2^*V_{\Omega}.
    \]
    Over $\Omega$ these two pullbacks are canonically identified because $V$ came from pulling back $\mathcal{E}_{\Omega}$; by finite-presentation-ness the isomorphism and its cocycle condition descend to $F$. Thus there exists an isomorphism of $R_F$-spaces
    \[
    \varphi_F:\mathrm{pr}_1^*V_F\overset{\sim}{\longrightarrow}\mathrm{pr}_2^*V_F
    \]
    satisfying the cocycle condition on triple overlaps.
    
    Using $\varphi_F$ we form the groupoid
    \[
    R_{V,F}\;:=\; \mathrm{pr}_1^*V_F \;\cong_{\varphi_F}\; \mathrm{pr}_2^*V_F
    \;\rightrightarrows\; V_F,
    \]
    with source and target maps $s,t:R_{V,F}\rightrightarrows V_F$ given by the two projections to the two factors of $\mathrm{pr}_i^*V_F\cong V_F\times_{U_F}R_F$. The cocycle condition guarantees the groupoid axioms hold for $R_{V,F}$.
    
    \bigskip
    
    Now we are able to construct the descent $\mathcal{E}_{F}$. We stackify the prestack $[V_{F}/R_{V,F}]^{pre}$ over the big fppf site to obtain a stack
    \[
    \mathcal{E}_{F} \;:=\; [\,V_F / R_{V,F}\,].
    \]
    
    Because $R_{V,F}\rightrightarrows V_F$ is a smooth equivalence relation, the stack quotient is an algebraic stack.
    
    \bigskip
    
    Finally we construct the map \(\mathcal{E}_{F}\to\mathcal X_{F}\) and check its properties. By our construction, the map $V_F\to U_F$ induces a morphism of stacks
    \[
    \pi_F:\mathcal{E}_{F}=[V_F/R_{V,F}]\longrightarrow \mathcal{X}_{F}.
    \]
    It is clear that all of the objects ($V_F$, $R_{V,F}$ and the isomorphism $\varphi_F$) were chosen so that their base change to $\Omega$ recovers the original objects $V$, $R$ and the canonical identification. Therefore $\pi_F\otimes_F \Omega=\pi$, thus the morphism $\pi_F$ is representable by algebraic spaces, and since $V_F\to U_F$ is finite \'etale, $\pi_F$ is finite \'etale as well, proving the lemma.
\end{proof}
Lemma \ref{lemma 6.5}, Lemma \ref{lemma 6.7} together give the full faithfulness and essential surjectivity of the functor, which completes the proof.
\end{proof}

\begin{lemma}
    $\pi_1^{top}([S_{K,s}](\C),s)\cong \Gamma_g$, $\pi_1([S_{K,s}]_{\C},s_{\C})\cong \widehat{\Gamma_g}$ and in general, $\pi_1([S_{K,s}]_{F},s_{F})\cong \widehat{\Gamma_g}$ for any algebraically closed field $F$ with characteristic $0$.
\end{lemma}
\begin{proof}
By Lemma \ref{lemma 6.3}, $[S_{K,s}](\C)=\Gamma_g\backslash X^{+}$. Since $\Gamma_g$ acts properly on the simply-connected manifold $X^{+}$, the Deck transformation group of the orbifold $[\Gamma_g \backslash X^{+}]$ is $\Gamma_g$, which is by definition the fundamental group of $[\Gamma_g \backslash X^{+}]$.

The second claim follows from \cite[Corollary 20.5]{noohi2005foundationstopologicalstacksi}. 

To prove the third claim, note that $[S_{K,s}]\cong [S_{K_0,s}/H]$ and that $S_{K_0,s}$ is normal, connected. So $[S_{K,s}]$ is connected and normal as well. Since $H$ is finite and that $S_{K_0,s}$ is Noetherian and separated, by Lemma \ref{lemma 6.8}, the diagonal of $[S_{K,s}]$ is also quasi-compact and quasi-separated. Thus, $[S_{K,s}]$ is a normal connected Noetherian DM stack, and the third claim follows from Lemma \ref{lemma 6.4}.
\end{proof}
\begin{definition}\cite[Definition 2.2]{klevdal2023compatibility}
    Let $G$ be an adjoint reductive group over $\mathbb{Q}$, and let $\Gamma$ be a discrete group. $\Gamma$ is said to be $G^{\mathrm{ad}}$-superrigid if for any algebraically closed field $\Omega$ of characteristic $0$ and any two homomorphisms $\rho_1,\rho_2:\Gamma\to G^{\mathrm{ad}}(\Omega)$ with Zariski-dense image, there exists an automorphism $\tau\in \mathrm{Aut}(G^{\mathrm{ad}}_{\Omega})$ such that $\rho_1=\tau(\rho_2)$.
\end{definition}

Now assume moreover that $G^{\mathrm{ad}}$ is $\Q$-simple and has real rank $\mathrm{rk}_{\mathbb{R}}(G^{\mathrm{ad}}_{\mathbb{R}})$ at least $2$. Since $(G,X)$ is a Shimura datum. \cite[2.3.4(a)]{Deligne_Shimura} shows that there exists a totally real field $F$ and an absolutely simple adjoint group $G^s$ over $F$ such that $G^{\mathrm{ad}}\cong \mathrm{Res}_{F/\Q}(G^s)$.
\begin{lemma}
    $\Gamma_g$ is superrigid.
\end{lemma}
\begin{proof}
    The group $\Gamma_g$ is an arithmetic subgroup of $G^{\mathrm{ad}}(\mathbb{Q})$ with $\operatorname{rk}_{\mathbb{R}}(G^{\mathrm{ad}}_{\mathbb{R}}) \ge 2$. Hence superrigidity follows from Margulis’s superrigidity theorem \cite[Theorem 2VIII.3.4(a)]{Margulis_discrete_subgroup}. The proof of \cite[Proposition~3.3]{klevdal2023compatibility} carries over as neatness of the compact open subgroup $K$ is not used there. 
    \end{proof}
\subsubsection{Good compactification, spreading out, and tame specialization map for certain DM stacks}
\label{Sec 6.1.2}

\paragraph{The good compactification of Shimura stacks}

\begin{lemma}
\label{lemma 6.8}
    Let $X$ be a scheme over another base scheme $S$ and let $H$ be a finite group scheme over $S$. For DM stacks of the form $[X/H]$, the diagonals are representable by schemes.
\end{lemma}
\begin{proof}
    By the universal property of stackification, the lemma will be proved if the following square is shown to be Cartesian:
    \begin{center}
    \begin{tikzcd}
    H\times_S X\arrow[r,]\arrow[d,]& X\times_S X \arrow[d,]\\
    \left[X/H\right]^{pre}\arrow[r,"\triangle"] & \left[X/H\right]^{\mathrm{pre}}\times_S \left[X/H\right]^{\mathrm{pre}}
    \end{tikzcd}
    \end{center}

    Since the objects in $\left[X/H\right]^{\mathrm{pre}}\times_S \left[X/H\right]^{\mathrm{pre}}$ are $(x,y)$ and that morphisms between $(x_1,y_1)$ and $(x_2,y_2)$ are $\{(h_1,h_2)\mid h_1,h_2\in H\}$, $\left[X/H\right]^{\mathrm{pre}}\times_S \left[X/H\right]^{\mathrm{pre}}\cong [X\times_S X/H\times_S H]^{pre}$, where $H\times_S H$ acts on $X\times_S X$ component-wise.

    Another computation shows that the objects of $[X/H]^{pre}\times_{(\left[X/H\right]^{\mathrm{pre}}\times_S \left[X/H\right]^{\mathrm{pre}})}(X\times_S X)$ are $(x_0,(x_1,x_2),(h_1,h_2))$ such that $h_1x_0=x_1$ and $h_2x_0=x_2$, where $x_0$ is an object in $[X/H]^{pre}$ and $(x_1,x_2)$ is an object in $X\times_S X$. There is no morphisms between $(x_0,(x_1,x_2),(h_1,h_2))$ and $(x_0',(x_1',x_2'),(h_1',h_2'))$ if $x_1\ne x_1'$ or $x_2\ne x_2'$ or $h_1'^{-1}h_1\ne h_2'^{-1}h_2$. Otherwise, the morphism is $h=h_1'^{-1}h_1=h_2'^{-1}h_2$.

    There is a map 
    \begin{align*}
        f:H\times_S X\to [X/H]^{pre}\times_{(\left[X/H\right]^{\mathrm{pre}}\times_S \left[X/H\right]^{\mathrm{pre}})}(X\times_S X).
    \end{align*}
    On the level of objects, $f$ maps $(h,x)$ to $(x,(x,hx),\mathrm{id},h)$. On the level of morphism, $f$ maps the identity to the identity.

    The full faithfulness of $f$ is clear by the description of the morphisms in $[X/H]^{pre}\times_{(\left[X/H\right]^{\mathrm{pre}}\times_S \left[X/H\right]^{\mathrm{pre}})}(X\times_S X)$. To prove that $f$ is essentially surjective, notice that any $(x_0,(x_1,x_2),(h_1,h_2))$ is equivalent to $(x_1,(x_1,x_2),(\mathrm{id},h_2h_1^{-1}))$. Therefore, $f$ is an isomorphism of prestacks and this proves the lemma.
    \end{proof}

To talk about good compactifications, we will follow the definition of normal crossings divisors (NCD in the following) as in \cite[Defintion 0BI9]{stacks-project}. Let $X$ be a smooth separated scheme of finite type over a field $k$ of characteristic $0$, and let $H$ be a finite group scheme over $k$ acting on $X$. $H$ is automatically flat and \'etale.

By Nagata compactification and Hironaka's resolution of singularities, $X$ has a good compactification, i.e., there exists a quasi-compact open embedding of $X$ into smooth proper variety $j:X\hookrightarrow\overline{X}$, such that $\overline{X}-X$ is an SNCD. Define $\overline{X}^{H}=\prod_{h\in H}X$, and $H$ acts on $\overline{X}^{H}$ by $h_0.(x_h)=x_{h_0^{-1}h}$ and define a map $f:X\to \overline{X}^{H}, x\to (j(h^{-1}x))_h$. The schematic image of $X$ in $\overline{X}^{H}$ gives an equivariant compactification of $X$.

Hironaka's idea on making the Nagata's compactification good is to use successive blow-ups to resolve the singularities, \cite[theorem 4.27]{reso_of_sing}. In the equivariant setting, denote by $\overline{X}$ the $H$-equivariant compactification of $X$. Then by successive blow-ups, I may assume that $\overline{X}$ is a good compactification of $X$. By \cite[proposition 3.9.1,theorem 3.36]{reso_of_sing}, group action on $X$ can be extended to its good compactification $\overline{X}$.

In order to define good compactification for DM stacks of the form, we make the following definition first.
\begin{definition}
\label{definition 6.12}
Let $S$ be a regular integral Noetherian scheme and let $\mathcal{Y}$ be a DM stack over $S$. A relatively normal crossings divisor (NCD) in $\mathcal{Y}$ is a closed substack $\mathcal{D}$ such that there exists an \'etale presentation $U\to \mathcal{Y}$ along which the pullback of $\mathcal{D}$ is a relatively strict normal crossings divisor in $Y$.
\end{definition}
Then we can define good compactification for stacks of the form $[X/H]$.

\begin{remark}
    A normal crossings divisor which is a strict normal crossings divisor after pulling back to a finite \'etale cover may not be a strict normal crossing divisor itself, as a nodal curve can have a finite \'etale cover $Y$ such that $Y$ is a union of two irreducible components which intersect transversely with each other.
\end{remark}

\begin{definition}
    A good compactification of DM stacks of the form $\mathcal{X}:=[X/H]$ over $k$ is a quasi-compact open immersion $\mathcal{X}\to \mathcal{Y}$ such that $\mathcal{Y}$ is smooth proper and that $\mathcal{Y}\to \mathcal{X}$ is an NCD as in Definition \ref{definition 6.12}.
\end{definition}

\begin{proposition}
    Let $X,H$ be the same as above, and let $\overline{X}$ denote the $H$-equivariant good compactification of $X$. Then $[X/H]\hookrightarrow[\overline{X}/H]$ is a good compactification.
\end{proposition}
\begin{proof}
Denote $[X/H]$ by $\mathcal{X}$ and $[\overline{X}/H]$ by $\overline{\mathcal{X}}$. 

To prove that $\mathcal{X}\hookrightarrow\overline{\mathcal{X}}$ is a quasi-compact open immersion, suffices to prove that the following is a Cartesian square:
\begin{equation}
\label{diagram 6}
    \begin{tikzcd}
    X\arrow[r,]\arrow[d,]& \overline{X} \arrow[d,]\\
    \mathcal{X}\arrow[hookrightarrow]{r}{i} &\overline{\mathcal{X}}
    \end{tikzcd}.
\end{equation}

First we study the fiber product of the corresponding prestacks. Still denote by $i$ the morphism $[X/H]^{pre}\to [\overline{X}/H]^{pre}$. The object in $\mathcal{X}\times_{\overline{\mathcal{X}}}X$ is $(x_0,x_1,h)$, where $x_0$ is an object in $[X/H]^{pre}$, $x_1$ is an object in $\overline{X}$ and $h.i(x_0)=x_1$. There is no morphism between $(x_0,x_1,h)$ and $(x_0',x_1',h')$ if $x_1\ne x_1'$. Otherwise, the morphism is $(h'^{-1}h,\mathrm{id})$. Now, the morphism $\iota: X\to \mathcal{X}\times_{\overline{\mathcal{X}}}\overline{X}$, which on the level of objects sends $x$ to $(x,x,id)$, is an isomorphism. Indeed, the fully faithfulness of $\iota$ is clear, while the essential surjectivity relies on the fact that $X\hookrightarrow\overline{X}$ is $H$-equivariant. This proves that the above diagram (\ref{diagram 6}) is Cartesian and thereby proves that $\mathcal{X}\hookrightarrow\overline{\mathcal{X}}$ is a quasi-compact open immersion. 

Clearly $\overline{\mathcal{X}}$ is smooth, since $\overline{X}$ is smooth and that $\overline{X}\to \overline{\mathcal{X}}$ is an \'etale presentation. To prove that $\overline{\mathcal{X}}$ is proper, first we show that it is separated, i.e., the diagonal morphism $\overline{\mathcal{X}}\to \overline{\mathcal{X}}\times_{k}\overline{\mathcal{X}}$ is proper. By Lemma \ref{lemma 6.8}, the diagonal morphism is \'etale locally the morphism $H\times_k \overline{X}\to \overline{X}\times_k \overline{X},(h,x)\to (x,hx)$. Since $H$ and $\overline{X}$ are both proper, the source and target are all proper. Since proper maps satisfies cancellation laws, the morphism $H\times_k \overline{X}\to \overline{X}\times_k \overline{X}$ is also proper. Since being proper is \'etale local on the target and preserved by arbitrary base change, the diagonal morphism is also proper. Thus, $\overline{\mathcal{X}}$ is separated over $k$. 

Since $\overline{X}\to \Spec{k}$ is locally of finite type and being locally of finite type is \'etale local in the \'etale topology, $\overline{\mathcal{X}}\to \Spec{k}$ is locally of finite type over $k$. Since $\overline{X}\to \Spec{k}$ is quasi-compact, $\overline{\mathcal{X}}\to \Spec{k}$ is also quasi-compact. Thus, $\overline{\mathcal{X}}\to \Spec{k}$ is of finite type.

Finally, since $\overline{X}\to \Spec{k}$ is universally closed, and that $\overline{X}\to \overline{\mathcal{X}}$ is surjective, $\overline{\mathcal{X}}\to \Spec{k}$ is universally closed. Therefore, $\overline{\mathcal{X}}$ is proper.

Denote by $\mathcal{D}$ the quotient stack $[D/H]$. It is left to prove that $\mathcal{D}=\overline{\mathcal{X}}-\mathcal{X}$ and that $\mathcal{D}$ is an NCD. The first claim is clear, as $X\to \overline{X}$ is an $H$-equivariant good compactification.

By a similar computation as before, the diagram
\begin{center}
    \begin{tikzcd}
    D\arrow[r,]\arrow[d,]& \overline{X} \arrow[d,]\\
    \mathcal{D}\arrow[hookrightarrow]{r}{} &\overline{\mathcal{X}}
    \end{tikzcd}
\end{center}
is also Cartesian. Therefore, by the definition of NCD in a DM stack, $\mathcal{D}$ is an NCD in $\overline{\mathcal{X}}$.
\end{proof}    

Having settled the good compactification, we can turn to the spreading out argument.

\paragraph{Spreading out quotient DM stacks}
\label{Sec 6.1.2.2}

Suppose that $k$ is a number field, and denote by $O_k$ its ring of integers. Pick a geometric point $\overline{x}$ of $X$. By \cite[theorem 4.2.1(i),theorem 4.2.1(ii)]{Spread}, we obtain an integral model $(\mathfrak{X},\overline{\mathfrak{X}},\mathfrak{D},\overline{\mathfrak{x}})$ over $\mathcal{O}_{k}[1/N]$ of $(X,\overline{X},D,\overline{x})$. We spread out the $H$ action first. We regard $H$ as a finite constant group scheme and the $H$ action can be regarded as a morphism $H\times_k \overline{X}\to \overline{X}$, with axioms regarded as the equality of certain morphisms. By \cite[theorem 4.2.1(i),theorem 4.2.1(iii)]{Spread}, these can all be spread to $\mathcal{O}_{k}[1/N]$, with a possibly larger $N$.

Moreover, by \cite[theorem 4.2.1(ii)]{Spread}, we are able to arrange it so that 
$\mathfrak{X}$ to be smooth, separated, geometrically connected and of finite type over $\mathcal{O}_k$ and that $\overline{\mathfrak{X}}$ also possess these properties and is in addition projective. Denote by $\overline{y}$ the image of $\overline{x}$ in $\mathcal{X}$ by the canonical \'etale presentation.

Denote by $\overline{y}$ the image of $\overline{x}$ in $\mathcal{X}$ by the canonical \'etale presentation. Then, the integral model $(\mathfrak{X},\overline{\mathfrak{X}},\mathfrak{D},\overline{\mathfrak{x}})$ is stable under the $H$ action, and the resulting pair of DM stacks $(\mathbb{X}:=[\mathfrak{X}/G],\overline{\mathbb{X}}:=[\overline{\mathfrak{X}}/G],\mathbb{D}:=[\mathfrak{D}/G],\overline{\mathfrak{y}})$ will be an integral model of $(\mathcal{X},\overline{\mathcal{X}},\mathcal{D},\overline{y})$, where $\overline{\mathfrak{y}}$ is the image of $\overline{\mathfrak{x}}$ in $\mathbb{X}$.

\paragraph{Tame ramification and tame specialization maps}
The goal of this subsection is also to define tame ramification and extend specialization map to the equivariant settings. In this subsection, the base scheme $S$ is an integral, excellent and pure-dimensional scheme. $X$ is a smooth geometrically connected separated scheme of finite type over $S$, $H$ a finite flat \'etale group scheme acting on $X$. Here, for a finite \'etale cover of schemes, say $Y\to X$, with $X$ having a good compactification, it is said to be tamely ramified if the discrete valuations associated with the generic points of $\overline{X}-X$ are tamely ramified in $k(Y)\mid k(X)$, which is equivalent to the curve-tameness. See \cite[\textsection 4]{tameness} for a precise definition of curve-tameness and proof for their equivalence.

\vskip 0.5cm
\begin{definition}
\label{definition 6.16}
    Let $\mathcal{X}$ be a smooth geometrically connected separated DM stack of finite type with a good compactification, and let $\mathcal{Y}\to \mathcal{X}$ be a finite \'etale cover. Then $\mathcal{Y}\to \mathcal{X}$ is said to be tamely ramified if for all morphisms from a regular separated curve $C$ of finite type over $S$ to $\mathcal{X}$, $C\to \mathcal{X}$, $\mathcal{Y}\times_{\mathcal{X}}C\to C$ is tamely ramified as a finite \'etale cover of curves.
\end{definition}
\begin{proposition}
    When $\mathcal{X}$ is a DM stack of the form $[X/H]$ over $S$, the following holds:
    \begin{enumerate}
        \item 
        \label{assertion 1} if $\mathcal{Y}\to \mathcal{X}$ is tamely ramified, then its pullback along $X\to \mathcal{X}$ is also tamely ramified as a finite \'etale cover of schemes. 
        \item 
        \label{assertion 2}for a finite \'etale cover $\mathcal{Y}\to \mathcal{X}$, if the characteristic of $S$ does not divide $|G|$, and if the pullback of $\mathcal{Y}\to \mathcal{X}$ along $X\to \mathcal{X}$ is tamely ramified, then $\mathcal{Y}\to \mathcal{X}$ is tamely ramified.
    \end{enumerate}
\end{proposition}
\begin{remark}
Assertion \ref{assertion 1} makes sense as finite \'etale morphisms are representable by schemes, since algebraic space locally quasi-finite and separated over a scheme is a scheme (see \cite{Moduli}, the criterion that a locally quasi-finite and separated algebraic space over a scheme is a scheme).
\end{remark}
\begin{proof}
First, we prove \ref{assertion 1}. Since $X\to \mathcal{X}$ is an \'etale presentation, $\mathcal{Y}\times_{\mathcal{X}}X\to X$ is also a finite \'etale morphism of schemes. Denote $\mathcal{Y}\times_{\mathcal{X}}X$ by $Y$. By the following diagram
\begin{equation*}
    \begin{tikzcd}
    Y\times_X C\arrow[r,]\arrow[d,]& C\arrow[d,]\\
    Y\arrow[r,]\arrow[d,]& X \arrow[d,]\\
    \mathcal{Y}\arrow[r,] &\mathcal{X}
    \end{tikzcd},
\end{equation*}
$Y\times_X C\to C$ is tamely ramified. Therefore, $Y\to X$ is curve tame and by \cite[theorem 4.4]{tameness}, this implies that $Y\to X$ is tamely ramified. 

To prove (\ref{assertion 2}), suppose that $\mathcal{Y}\to \mathcal{X}$ pulls back to a tamely ramified cover of schemes along $X\to \mathcal{X}$. Let $C\to \mathcal{X}$ be a morphism from a regular separated curve of finite type over $S$ to $\mathcal{X}$. We have the diagram:
\begin{equation*}
    \begin{tikzcd}
    \widetilde{C}'\arrow[rrr,]\arrow[ddd,]\arrow[dr,]& & & C'\arrow[dl,]\arrow[ddd,]\\
    &Y\arrow[r,]\arrow[d,]& X \arrow[d,]& \\
    &\mathcal{Y}\arrow[r,] &\mathcal{X}&\\
    \widetilde{C}\arrow[rrr,]\arrow[ur,]& & & C\arrow[ul,]\\
    \end{tikzcd}
\end{equation*}
where the four trapezoids and the smaller and larger square are all Cartesian. The degree of the finite \'etale morphism $C'\to C$ is $|G|$, and since characteristic of $S$ does not divide $|G|$, $C'\to C$ is at most tamely ramified. Since $\widetilde{C}'\to C'$ is also tamely ramified, that implies $\widetilde{C}\to C$ is also tamely ramified. Since this holds for all such $C\to \mathcal{X}$, $\mathcal{Y}\to \mathcal{X}$ is tamely ramified.

\end{proof}

By the above definition, if the characteristic of $S$ is $0$, any finite \'etale cover is tamely ramified. 

For a finite group $H$ acting on a smooth geometrically connected separated scheme $X$ of finite type over a field $k$ of positive characteristic $p$ not dividing $|H|$, the above definition shows that $X\to [X/H]$ is a tamely ramified Galois cover with Galois group $H$. Indeed, since $X\times_{[X/H]}X\cong H\times X$, the pull back of $X\to [X/H]$ along $X\to [X/H]$ is $H\times X\to X$. Combined with the fact that the $H$ action extends to $\overline{X}$ and that $H$ acts by automorphism, this shows that $X\to [X/H]$ is a tamely ramified cover over $k$.
\begin{lemma}
\label{lemma 6.19}
    Let $T\to R$ be a morphism of base schemes. Let $X$ be a connected scheme over $R$ such that $X_T:X\times_R T$ is also connected. Suppose that $H$ is a constant finite group scheme acting on $X$, where both $H$ and the action are defined over $R$. Let $x$ be a geometric point of $X$ and $x'$ a geometric point of $X_T$. Let $y,y'$ be the image of $x,x'$ in $[X/H]_R,[X/H]_T$ respectively. Then there is a commutative diagram with exact rows, with the rightmost vertical arrow being an isomorphism:
    \begin{equation}
    \label{diagram 7}
    \begin{tikzcd}
    1\arrow[r,]&\pi_1(X_{T},x)\arrow[r,]\arrow[d,]& \pi_1([X/H]_{T},y) \arrow[r,]\arrow[d,]& H \arrow[r,]\arrow[d,]& 1\\
    1\arrow[r,]&\pi_1(X_R,x')\arrow[r,]&\pi_1([X/H]_R,y')\arrow[r,]& H \arrow[r,]& 1.
    \end{tikzcd}    
\end{equation}
\end{lemma}
\begin{proof}
    Clearly $X\to [X/H]$ is an \'etale presentation. Since $X\times_{[X/H]}X\cong H\times X$, and that $H$ is a finite group scheme, $X\to [X/H]$ is also finite. Therefore by the assumptions, $X\to [X/H]$ is a connected finite Galois cover of $[X/H]$ with Galois group $H$, whether the base scheme is $R$ or $T$. These give the exactness of the top and the bottom row.

    We explain how the rightmost vertical arrow comes. The base change morphism $T\to R$ induces a map from the deck transformation group of $X_R\to [X/H]_R$ to the deck transformation group of $X_{T}\to [X/H]_{T}$. Since $X_T$ is connected, the map is injective. Since $H$ is a finite constant group scheme, the map is an isomorphism. Thus, we can take its inverse and this gives the rightmost vertical arrow, which also shows that the rightmost vertical arrow is an isomorphism.
\end{proof}
\begin{lemma}
    Same notations as in Lemma \ref{lemma 6.19}. When $X\to [X/H]$ is moreover tamely ramified, the diagram with the \'etale fundamental group replaced by the tame \'etale fundamental group is also commutative, has exact rows, and has the rightmost vertical map as an isomorphism.
\end{lemma}
\begin{proof}
    Apply the same idea as in the proof of Lemma \ref{lemma 6.19} since now $X\to [X/H]$ is tame.
\end{proof}

Let $A$ be a strictly henselian discrete valuation ring which is also a G-ring, then $\Spec{A}$ is integral, excellent and pure dimensional. Let $\overline{\eta}$ be the geometric point over the generic point $\eta$ of $\Spec{A}$, and let $\overline{s}$ be the geometric point over the special point $s$ of $\Spec{A}$. By \cite[\textsection 2]{klevdal2023compatibility}, the main point of the tame specilization map for a smooth geometrically connected variety $X$ over $A$ has two ingredients, where $x_{\overline{\eta}}$ is a geometric point of $X_{\overline{\eta}}$ and $x_{\overline{s}}$ is a geometric point of $X_{\overline{s}}$:
\begin{enumerate}
    \item[(i)] the map $\pi_1^t(X_{\overline{\eta}},x_{\overline{\eta}})\to \pi_1^t(X,x_{\overline{\eta}})$ induced by the canonical inclusion $X_{\eta}\to X$ is surjective; 
    \item[(ii)] the map $\pi_1^t(X_{\overline{s}},x_{\overline{s}})\to \pi_1^t(X,x_{\overline{s}})$ induced by the canonical inclusion $X_{\overline{s}}\to X$ is an isomorphism; 
\end{enumerate}

We wish to establish similar claims for DM stacks $[X/H]$ defined over ring $A$ and its generic and special fibers.

Let $y_{\overline{\eta}}$ be the image of $x_{\overline{\eta}}$ by $X_{\overline{\eta}}\to [X/H]_{\overline{\eta}}$, then $y_{\overline{\eta}}$ is also a geometric point of $[X/H]$.
Since the pull back of $[X/H]_{\eta}\to [X/H]$ along $X\to [X/H]$ is $X_{\eta}$, there is a diagram consisting of two exact rows:
\begin{center}
    \begin{tikzcd}
    1\arrow[r,]&\pi_1^t(X_{\overline{\eta}},x_{\overline{\eta}})\arrow[r,]\arrow[d,]& \pi_1^t([X/H]_{\overline{\eta}},y_{\overline{\eta}}) \arrow[r,]\arrow[d,]& H \arrow[r,]\arrow[d,]& 1\\
    1\arrow[r,]&\pi_1^t(X,x_{\overline{\eta}})\arrow[hookrightarrow]{r}{} &\pi_1^t([X/H],y_{\overline{\eta}})\arrow[r,]& H \arrow[r,]& 1.
    \end{tikzcd}
\end{center}

By the first part of properties of tame specialization map in the scheme case, the leftmost vertical arrow is a surjection. Since the rightmost arrow is an isomorphism, by snake lemma, the middle map is a surjection.

Similarly, for the special fiber, let $y_{\overline{s}}$ be the image of $x_{\overline{s}}$ by $X_{\overline{s}}\to [X/H]_{\overline{s}}$, then $y_{\overline{s}}$ is also a geometric point of $[X/H]$. Then there is a diagram consisting of two exact rows:

\begin{center}

    \begin{tikzcd}
    1\arrow[r,]&\pi_1^t(X_{\overline{s}},x_{\overline{s}})\arrow[r,]\arrow[d,]& \pi_1^t([X/H]_{\overline{s}},y_{\overline{s}}) \arrow[r,]\arrow[d,]& H \arrow[r,]\arrow[d,]& 1\\
    1\arrow[r,]&\pi_1^t(X,x_{\overline{s}})\arrow[hookrightarrow]{r}{} &\pi_1^t([X/H],y_{\overline{s}})\arrow[r,]& H \arrow[r,]& 1.
    \end{tikzcd}
\end{center}

The rightmost vertical arrow is defined similarly as before, where it is the inverse of the specialization morphism from the deck transformation group of $X\to [X/H]$ to the deck transformation group of $X_{\overline{s}}\to [X/H]_{\overline{s}}$.

The middle arrow is an isomorphism since both the leftmost one and the rightmost one are isomorphisms.

In conclusion, there is a set of similar statements in the DM stack $[X/H]$ case:
\begin{enumerate}
    \item[(i)] the map $\pi_1^t([X/H]_{\overline{\eta}},y_{\overline{\eta}})\to \pi_1^t([X/H],y_{\overline{\eta}})$ induced by the canonical inclusion $[X/H]_{\eta}\to [X/H]$ is surjective; 
    \item[(ii)] the map $\pi_1^t([X/H]_{\overline{s}},y_{\overline{s}})\to \pi_1^t([X/H],y_{\overline{s}})$ induced by the canonical inclusion $[X/H]_{\overline{s}}\to [X/H]$ is an isomorphism; 
\end{enumerate}

\subsubsection{Spreading out the canonical $l$-adic local system to the integral model}
Back to the Shimura stack setting, we utilize Section \ref{Sec 6.1.2} to derive some results here. As in Section \ref{Sec 6.1.1}, let $(G,X)$ be a Shimura datum, $K$ a compact open subgroup of $G(\mathbb{A}_f)$, not necessarily neat and let $K_0$ be a neat subgroup of $K$ of finite index. Let $S_{K_0,s}$ be a geometrically connected component of $\mathrm{Sh}_{K_0}(G,X)$, then $K/K_0$ acts on the geometrically connected components of $\mathrm{Sh}_{K_0}(G,X)$. Denote by $H$ the stabilizer of $S_{K_0,s}$ in $K/K_0$. As in Section \ref{Sec 6.1.1}, the Shimura stack $[S_{K,s}]$ is defined as $[S_{K_0,s}/H]$. Let $F$ be a fixed number field, $G^{\mathrm{ad}}$ an adjoint semisimple algebraic group, $\Gamma$ a discrete $G^{\mathrm{ad}}$-superrigid group and let $\lambda$ be finite places of $\overline{\Q}$. 

Let $s$ be a point in $[S_{K,s}]$ which is the image of some point $s_1$ in $S_{K_0,s}$, and let $\overline{s}$ be a geometric point lying over $s$. By Section \ref{Sec 6.1.2.2}, the integral model $\mathbb{S}_{K,s}$ of $[S_{K,s}]$ is constructed by choosing an integral model $\mathcal{S}_{K_0,s}$ of $S_{K_0,s}$ which is smooth, geometrically connected, separated, and of finite type, extending $H$-action onto $\mathcal{S}_{K_0,s}$, and then define $\mathbb{S}_{K,s}$ to be the stack quotient of $\mathcal{S}_{K_0,s}$ by $H$.
\begin{proposition}
\label{proposition 6.21}
    Let $\rho$ be the canonical $l$-adic local system on $[S_{K,s}]$, constructed as in Section \ref{Sec 6.1.1}, i.e., the adjoint representation:
    \begin{align*}
    \rho_{K,s,\lambda}:\pi_1([S_{K,s}],\overline{s})\to K\hookrightarrow G(\mathbb{A}_f)\twoheadrightarrow G^{\mathrm{ad}}(\Q_\ell)\to G^{\mathrm{ad}}(\overline{\Q_{\lambda}}).
    \end{align*}for $\lambda$ lying over $l$.
    Then it extends to a family of representations (still denoted by the same notation):
    \begin{align*}
\rho_{\lambda}:\pi_1(\mathbb{S}_{K,s,\mathcal{O}_F[1/Nl]},\overline{s})\to G^{\mathrm{ad}}(\overline{\Q_{\lambda}}),
\end{align*} where $\mathbb{S}_{K,s}$ is a model for $[S_{K,s}]$ over $\mathcal{O}_{E_{K,s}}$.
\end{proposition}
\begin{proof}
    First, by the construction of the canonical $l$-adic local system on $[S_{K,s}]$ and $S_{K_0,s}$, it is clear that the restriction of $\rho_{K,s,\lambda}$ to $\pi_1(S_{K_0,s})$ is the same as the canonical $l$-adic local system constructed in \cite[\textsection 3.1]{klevdal2023compatibility}. Denote the restricted representation by $\rho^0_{K_0,s,\lambda}$. 

    It is proved in \cite[proposition 2.3]{klevdal2023compatibility} that $\rho^0_{K_0,s,\lambda}$ descent to a family of representations, which we still denote by $\rho^0_{K_0,s,\lambda}$:
\begin{align*}
\rho^0_{K_0,s,\lambda}:\pi_1(\mathcal{S}_{K_0,s,\mathcal{O}_{E_{K_0,s}}[1/Nl]},\overline{s})\to G^{\mathrm{ad}}(\Q_{\lambda}),
\end{align*}
where $N$ is independent of $\lambda$.

In the spreading out process, we require $\mathcal{S}_{K_0,s}$ to be connected, so we can apply Lemma \ref{lemma 6.19} to conclude that the kernel of $\pi_1(S_{K_0,s},\overline{s_1})\to \pi_1(\mathcal{S}_{K_0,s},\overline{s_1})$ is the same as that of $\pi_1([S_{K,s}],\overline{s}) \to\pi_1(\mathbb{S}_{K,s},\overline{s})$. Since $\rho^0_{K_0,s,\lambda}$ factors through the leftmost vertical map in Diagram \ref{diagram 7} (with $X=\mathcal{S}_{K_0,s}$, $T=E_{K_0,s}$, $R={O}_{E_{K_0,s}}[1/Nl]$), $\rho_{K,s,\lambda}|_{\pi_1([S_{K,s}]_{E_{K_0,s}},\overline{s})}$ factors through the middle map in the same diagram. Finally, we notice that $E_{K_0,s}/E_{K,s}$ is a finite extension of fields, where only finitely many primes ramify. So if we enlarge $N$ so that it includes those ramified primes as well, then $\rho_{K,s,\lambda}$ extends to a local system on $\mathbb{S}_{K,s}$ defined over ${O}_{E_{K,s}}[1/Nl]$. Since all $\rho^0_{K_0,s,\lambda}$ have such an extension, so do all $\rho_{K,s,\lambda}$'s.

\end{proof}

\begin{definition}
    Let $X$ be a DM (resp. an algebraic stack) of finite type over $\mathcal{O}_F$ for some number field $F$ and $x\in X$ a closed point. Let $G$ be a reductive algebraic group over $\Q$ and $\{\rho_{\lambda,\iota_\lambda}\}_{\lambda,\iota_\lambda}$ be family of representations $\rho_{\lambda,\iota_\lambda}:\pi_1(X,\overline{x})\to G(\overline{\Q_\ell})$ for various $\lambda$ and $F$-embeddings $\iota_\lambda:\overline{\Q}\to \overline{\Q}_{\lambda}$. Then $\{\rho_{\lambda,\iota_\lambda}\}_{\lambda,\iota_\lambda}$ is said to be compatible if there exists an $N$ and that for all closed point $x$ with $\mathrm{char}(x)\nmid N$, there exists a number field $E$ independent of $x$ with a fixed embedding $\iota: E\to \overline{\Q}$ and a $\gamma_x\in [G//G](E)$, such that for all $\lambda\nmid \mathrm{char}(x)$, $\gamma_x=[\rho_{\lambda,\iota_\lambda}(\mathrm{Frob}_x)]$ in $[G//G](\overline{\Q}_{\lambda})$, where $\gamma_x$ is identified with its image in $[G//G](\overline{\Q}_{\lambda})$ by the restriction of $\iota_{\lambda}$ to $E$. 
\end{definition}
Let $N'=N\cdot |H|$, with $N$ in Proposition \ref{proposition 6.21}. For two distinct places $\lambda,\lambda'$, let $\ell,\ell'$ be their residual characteristics respectively. Let $v$ be a prime in $\mathcal{O}_F[1/N'\ell\ell']$. Denote by $\mathbb{S}_{K,s,v}$ the $v$-fiber of $\mathbb{S}_{K,s}$ and by $\rho_{K,s,\lambda,v}$ the pullback of $\rho_{K,s,\lambda}$ to $\pi_1(\mathbb{S}_{K,s,v})$, via the canonical morphism $\pi_1(\mathbb{S}_{K,s,v})\to \pi_1(\mathbb{S}_{K,s})$.
\begin{proposition}
    Let $\lambda,\lambda',v$ be as above. There exists a $\tau\in \mathrm{Aut}(G_{\overline{\Q}_{\lambda'}}^{\mathrm{ad}})$ such that such that $\tau(\rho_{K,s,\lambda',v})$ is compatible with $\rho_{K,s,\lambda,v}$. This $\tau$ may be dependent on $\lambda,\lambda'$ and $v$.
\end{proposition}
\begin{proof}
    Denote by $\kappa(v)$ the residue field of $v$, and $\mathbb{S}_{K,s,\overline{v}}$ the fiber product $\mathbb{S}_{K,s,v}\times_{\kappa(v)}\overline{\kappa(v)}$.
    By Theorem \ref{theorem 4.2}, $\rho_{K,s,\lambda,v}$ has a companion $\rho_{K,s,\lambda'\rightsquigarrow\lambda,v}$ whose image is Zariski dense in $G_{\overline{\Q}_{\lambda'}}^{\mathrm{ad}}$. By Definition \ref{definition 6.16} and \cite[Théorème 9.8]{Deligne_l_function}, $\rho_{K,s,\lambda'\rightsquigarrow\lambda,v}$ is also tame. Denote by $\rho_{K,s,\lambda'\rightsquigarrow\lambda,\overline{v}}$ the restriction of $\rho_{K,s,\lambda'\rightsquigarrow\lambda,v}$ to $\pi_1(\mathbb{S}_{K,s,\overline{v}})$. Suppose that its monodromy group is $H$, a subgroup of $G_{\overline{\Q}_{\lambda'}}^{\mathrm{ad}}$. Since $G_{\overline{\Q}_{\lambda'}}^{\mathrm{ad}}$ is adjoint, it is a direct product of adjoint simple algebraic groups. Since $H$ is normal subgroup, it is also a product of some of those adjoint simple factors. If $H\ne G_{\overline{\Q}_{\lambda'}}^{\mathrm{ad}}$, then by the homotopy exact sequence for algebraic stacks, the Frobenius of the base field $\kappa(v)$ is Zariski-dense in $G_{\overline{\Q}_{\lambda'}}^{\mathrm{ad}}/H$. Since $G_{\overline{\Q}_{\lambda'}}^{\mathrm{ad}}/H$ is semisimple and adjoint, all abelian subgroups are trivial. Therefore, $H=G_{\overline{\Q}_{\lambda'}}^{\mathrm{ad}}$ and $\rho_{K,s,\lambda'\rightsquigarrow\lambda,\overline{v}}$ has Zariski dense image in $G_{\overline{\Q}_{\lambda'}}^{\mathrm{ad}}$. Let $\overline{s}_v$ be the specialization of $s_{\mathcal{O}_{E_{K,s}}}$ in $[S_{K,s,\overline{v}}]$. Pulling back $\rho_{K,s,\lambda'\rightsquigarrow\lambda,\overline{v}}$ to $\pi_1^{top}([S_{K,s}](\C),s)$, there is a representation:
    \begin{align*}
\rho_{K,s,\lambda'\rightsquigarrow\lambda,\overline{v}}^{\mathrm{top}}:\pi_1^{top}([S_{K,s}](\C),s)\to \pi_1([S_{K,s}]_{\overline{\Q}},\overline{s})\xrightarrow{sp}\pi_1^t([S_{K,s,\overline{v}}],\overline{s}_v)\xrightarrow{\rho_{K,s,\lambda'\rightsquigarrow\lambda,\overline{v}}} G^{\mathrm{ad}}(\overline{\Q}_{\lambda'}).
    \end{align*} with Zariski dense image in $G_{\overline{\Q}_{\lambda'}}^{\mathrm{ad}}$ since $\pi_1([S_{K,s}]_{\overline{\Q}},\overline{s})\cong \pi_1([S_{K,s}]_{\C},\overline{s})$ is the profinite completion of $\pi_1^{top}([S_{K,s}](\C),s)$. Denote by $\rho_{K,s,\lambda',\overline{v}}^{\mathrm{top}}$ a similar pullback of $\rho_{K,s,\lambda',\overline{v}}$ via the same map. By the superrigidity of $\Gamma_g=\pi_1^{top}([S_{K,s}](\C),s)$, there is a $\tau\in \mathrm{Aut}(G_{\overline{\Q}_{\lambda'}}^{\mathrm{ad}})$, such that $\tau(\rho_{K,s,\lambda',\overline{v}}^{\mathrm{top}})=\rho_{K,s,\lambda'\rightsquigarrow\lambda,\overline{v}}^{\mathrm{top}}$. As $G_{\overline{\Q}_{\lambda'}}^{\mathrm{ad}}$ has trivial center, and since both are descent of $\rho_{K,s,\lambda'\rightsquigarrow\lambda,\overline{v}}=\tau(\rho_{K,s,\lambda',\overline{v}})$, $\tau(\rho_{K,s,\lambda',v})=\rho_{K,s,\lambda'\rightsquigarrow\lambda,v} $, and thus $\tau(\rho_{K,s,\lambda',v})$ is compatible with $\rho_{K,s,\lambda,v}$.
\end{proof}
In \cite[\textsection 3.3]{klevdal2023compatibility}, it is shown that there exists an $N''$ such that for any two places $\lambda,\lambda'$, for any  prime $v$ in $\mathcal{O}_{E_{K_0,s}}[1/N''\ell\ell']$, there is a point $y\in \mathcal{S}_{K_0,s}(\mathcal{O}_{C_y}[1/N_y])$, where $C_y/F$ is a finite extension, and $N_y$ is some multiple of $N''$ such that $N_y$ is not divisible by $v$, and the following two conditions hold:
\begin{enumerate}
    \item The induced Galois representations $\rho_{\lambda,y}: \pi_1(C_y) \to  G^{\mathrm{ad}}(\overline{\Q}_{\lambda})$ and $\rho_{\lambda',y}: \pi_1(C_y) \to  G^{\mathrm{ad}}(\overline{\Q}_{\lambda'})$ are compatible at any place $v_y$ of $C_y$ above $v$. i.e., with $F_{v_y}$ being the Frobenius above the prime $v_y$, $\rho_{\lambda,y}(F_{v_y})$ and $\rho_{\lambda',y}(F_{v_y})$ have semisimple parts $t_{\lambda,v_y}$ and $t_{\lambda',v_y}$ whose conjugacy classes are defined over $\overline{\Q}$ and in $G^{\mathrm{ad}}(\overline{\Q})$ define the same conjugacy class. 
    \item The element $t_{\lambda,v_y} \in G^{\mathrm{ad}}(\overline{\Q}_{\lambda})$ is not fixed by any automorphism of $G^{\mathrm{ad}}_{\overline{\Q}_{\lambda}}$ with non-trivial image in the outer automorphism group $\mathrm{Out}(G^{\mathrm{ad}}_{\overline{\Q}_{\lambda}})$. 
\end{enumerate}
Viewing $y$ as a point of $\mathbb{S}_{K,s}\bigl(\mathcal{O}_{C_y}[1/N_y]\bigr)$, possibly with a different choice of $N_y$, and noting that $E_{K_0,s}/E_{K,s}$ is a finite extension, we see that the same conditions also hold in the case of Shimura stacks.
\begin{proposition}
For any two $\lambda,\lambda'$, and any $v$ in $\mathcal{O}_F[1/N''\ell\ell']$ such that the above two conclusions hold, $\rho_{K,s,\lambda,v}$ is compatible with $\rho_{K,s,\lambda',v}$.
\end{proposition}
\begin{proof}
    Since the proof \cite[Proposition 2.8]{klevdal2023compatibility} is algebraic-group-theoretical, the same proof carries over.
\end{proof}

\begin{theorem}
\label{theorem 6.25}
Let $(G,X)$ be a Shimura datum such that $G^{\mathrm{ad}}$ is $\Q$-simple, $Z_G(\Q)$ is a discrete subgroup of $Z_G(\A_f)$ and $\mathrm{rk}_{\R}(G_\R) \geq 2$. Let $K \subset G(\A_f)$ be a compact open subgroup (not necessarily neat), and let $s \in \Sh(\C)$, $[S_{K,s}]$ be as above. Then there is an integer $N$ and an integral model $\mathbb{S}_{K,s}$ over $\mathcal{O}_{E_{K_0,s}}[1/N]$ such that the canonical local system valued in the adjoint group of $G$: $\pi_1(\mathbb{S}_{K,s}, s) \to G^{\mathrm{ad}}(\overline{\Q_\ell})$ is compatible in the following sense: for all closed points $x \in \mathbb{S}_{K,s}[1/\ell]$, the class of $\rho_\ell(\mathrm{Frob}_x)$ in $[G^{\mathrm{ad}}/G^{\mathrm{ad}}](\overline{\Q}_\ell)$ lies in $[G^{\mathrm{ad}}/G^{\mathrm{ad}}](\overline{\Q})$ and is independent of $\ell$ (not equal to the residue characteristic of $x$).
\end{theorem}

\subsection{The reductive case}
Here, we treat the more general case where the local system is valued in a reductive group, with the main theorem being Theorem \ref{theorem 6.29}. The proof follows the spirit of \cite{patrikis2025compatibilitycanonicalelladiclocal}, where an important element is the special points lifting results from \cite{bakker2025integralcanonicalmodelsexceptional}.
\subsubsection{Extension of abelian compatibility to the stack case}
\begin{lemma}
Let $(G,X)$ be a Shimura datum satisfying SV5 as before and $K$ a compact open subgroup in $G(\mathbb{A}_f)$ (not necessarily neat). Let $s = [x,a] \in \mathrm{Sh}(\mathbb{C})$ be a special point on a Shimura stack $[S_{K,s}]$ over a number field $E(s_{K})$. Then the canonical local system $\rho_{K,s}$ specialized to $s_{K}$ defines a compatible system of $G(\mathbb{Q}_\ell)$-representations in the following sense: there exists an integer $N$, such that for any choice of Frobenius $\mathrm{Frob}_v$ at a place $v|p$ $(p\nmid N)$ of $E(s_{K_0})$, and for all $\ell \neq p$, $\rho_{\ell,s}(\mathrm{Frob}_v) \in G(\mathbb{Q}_\ell)$ is $G(\mathbb{Q}_\ell)$-conjugate to a value independent of $\ell$ in $G(\mathbb{Q}_\ell)$.
    
\end{lemma}
\begin{proof}
    The injective map in \cite[theorem 1.6]{taelman2018complex} maps special points $s$ in Shimura stacks to canonical local systems specialized at $x$, which follows from the definition of the canonical local system on Shimura stack and the construction of the isomorphism in \cite[lemma 3.4]{taelman2018complex}.

    \begin{center}
    \begin{tikzcd}
    \mathrm{Gal}_F\arrow[r,]\arrow[d,"\overline{\rho}"]& \mathrm{Gal}_E \arrow[d,"rec"]\\
    \mathcal{U}/\mathcal{U}_0\arrow[r,] & T(\Q)\backslash T(\A_f)/\mathcal{U}_0
    \end{tikzcd}
    \end{center}
    
    \begin{center}
    \begin{tikzcd}
    \mathcal{U}\arrow[r,]\arrow[d,]& T(\Q)\backslash T(\A_f) \arrow[d,]\\
    \mathcal{U}/\mathcal{U}_0\arrow[r,] & T(\Q)\backslash T(\A_f)/\mathcal{U}_0
    \end{tikzcd}
    \end{center}
    Since the first diagram is commutative and the second diagram is Cartesian, they imply that the specialized canonical local system comes from a fiber product of a Galois representation of finite image and an inflation of some specialized canonical local system on a Shimura variety. Thus, by possibly enlarging $N$, the set of bad primes, the canonical local system $\rho_{K_0,s}$ specialized to $s_{K_0}$ defines a compatible system of $G(\mathbb{Q}_\ell)$-representations.
\end{proof}
\subsubsection{G-equivariant version of lifting special points}
Since the slopes of Frobenius on the $F$-crystal do not change when one extends the base field, being $\mu$-ordinary is a property stable under finite base field extension, and we are able to make the following definitions:
\begin{definition}
A point $x$ of a Shimura stack $\mathfrak{S}$ is called $\mu$-ordinary if,
for some étale presentation $S \to \mathfrak{S}$ by a
Shimura variety, every point of $S$ lying above $x$ is $\mu$-ordinary.
\end{definition}

By \cite[theorem 1.5]{bakker2025integralcanonicalmodelsexceptional}, we obtain the following:
\begin{corollary}
    The $\mu$-ordinary locus in the integral canonical model of a Shimura stack is open and dense. All finite field valued $\mu$-ordinary points have a canonical lift as a special point, which is valued in the Witt ring of that finite field.
\end{corollary}

\begin{theorem}
\label{theorem 6.29}
    Let $(G,X)$ be a Shimura datum such that $Z_G(\mathbb{Q})$ is a discrete subgroup of $Z_G(\mathbb{A}_f)$, let $K \subset G(\mathbb{A}_f)$ be a compact open subgroup (not necessarily neat), and let $s \in \mathrm{Sh}(\mathbb{C})$, $[S_{K,s}]$ be as above. Assume that for all $\mathbb{Q}$-simple factors $H$ of $G^{\mathrm{ad}}$, $\operatorname{rk}_{\mathbb{R}}(H_{\mathbb{R}}) \geq 2$. Then there is an integer $N$ and an integral model $\mathbb{S}_{K,s} \text{over } \mathcal{O}_{E_{K,s}}[1/N]$ such that for all closed points $x \in \mathbb{S}_{K,s}[1/\ell]$, the class of $\rho_\ell(\mathrm{Frob}_x)$ in $[G/G](\mathbb{Q}_\ell)$ lies in $[G/G](\mathbb{Q})$ and is independent of $\ell$ (not equal to the residue characteristic of $x$).
\end{theorem}
\begin{proof}
We'll use an open and dense substack in place of the "Dirichlet density 1 subscheme" as in \cite[Theorem 4.4]{patrikis2025compatibilitycanonicalelladiclocal}. By previous lemma, we know that for a geometrically integral algebraic stack of finite presentation over an arbitrary finite field $k$, the \'etale fundamental group of the whole stack is a quotient of the \'etale fundamental group of an open substack. So any two representations of the \'etale fundamental group of the whole stack, which agree on the \'etale fundamental group of an open substack, will actually be the same.

Apart from that, the other ingredients in the proof of \cite[Theorem 4.4]{patrikis2025compatibilitycanonicalelladiclocal} that are relevant to the scheme structure of the Shimura variety is the Čebotarev density theorem, whose extension to the stack case is proved in \cite[Proposition 4.6] {Zheng_2018}. 
\end{proof}

\newpage
\bibliographystyle{plain}
\bibliography{ref}
\end{document}